\documentclass [12pt,a4paper,reqno]{amsart}
\usepackage{etex}
\usepackage{amsmath, amssymb}
\usepackage{extarrows}
\usepackage{xypic}
\usepackage{amsmath,amsthm}
\usepackage{amssymb,amscd}
\usepackage{latexsym}
\usepackage{graphicx}

\usepackage{pst-all}

\usepackage{pstricks}

\usepackage{tikz}
\usetikzlibrary{positioning}

\input xy
\xyoption{all}

\newcommand{\etype}[1]{\renewcommand{\labelenumi}{(#1{enumi})}}
\def\eroman{\etype{\roman} \dispace}
\def\ealph{\etype{\alph} \dispace }
\def\dispace{\setlength{\itemsep}{2pt}}
\newcommand{\ds}[1]{\ {#1} \ }
\newcommand{\dss}[1]{\quad {#1} \quad }

\def\Iff{\ \Leftrightarrow \ }

\newcommand{\Qthmref}[1]{\cite[Theorem~{#1}]{QF1}}

\newcommand{\Qpropref}[1]{\cite[~Proposition~{#1}]{QF1}}
\newcommand{\Qsecref}[1]{\cite[~\S{#1}]{QF1}}

\newcommand{\Qdefref}[1]{\cite[Definition~{#1}]{QF1}}

\def\sm{\setminus}
\def\00{ \{ 0 \}}

\def\tlA{\widetilde{A}}

\def\tlC{\widetilde{C}}
\def\tlT{\widetilde{T}}

\def\tlD{\widetilde{D}}

\def\htT{\widehat{T}}
\def\htQ{\widehat{Q}}

\def\sm{\setminus}

\def\Dir{\dss \Rightarrow}

\def\nucong{\cong_\nu}
\def\noi{\noindent}

\def\pSkip{\vskip 1.5mm \noindent}
\def\semiring0{semiring$^{\dagger}$}

\newcommand{\conv}{\operatorname{conv}}
\newcommand{\Ray}{\operatorname{Ray}}
\newcommand{\ray}{\operatorname{ray}}
\newcommand{\CS}{\operatorname{CS}}
\newcommand{\an}{\operatorname{an}}
\newcommand{\QL}{\operatorname{QL}}
\newcommand{\sat}{\rm sat}
\newcommand{\Max}{\rm Max}

\def\Max{\mathop{\rm Max}\limits}

\def\pipe {\ds |}

\def\tT{\mathcal T}

\def\tG{\mathcal G}

\newtheorem{thm}{Theorem} [section]
\newtheorem*{thm*}{Theorem}
\newtheorem{cor}[thm]{Corollary}
\newtheorem{lem}[thm]{Lemma}

\newtheorem{prop}[thm]{Proposition}

\newtheorem*{claim*} {Claim}
\newtheorem{proc}[thm]{Procedure}

\newtheorem*{theorem13.5'} {Theorem 13.5$'$}
\newtheorem{acknowledgment*}[thm] {Acknowledgment}

\newtheorem{examp}[thm]{Example}

    \newtheorem*{remarks*} {Remarks}
 \newtheorem{comm}[thm]{Comment}
 \newtheorem*{remark*}{Remark}
 \newtheorem{defn}[thm]{Definition}
\newtheorem{construction}[thm]{Construction}

\newtheorem{schol}[thm]{Scholium}
\newtheorem{notation}[thm]{Notation}

\newtheorem*{notation*} {Notation}

\newtheorem{rem}[thm]{Remark}

 \renewcommand{\sectionmark}[1]{}

\newcommand{\bfem}[1]{\textbf{#1}}

\newcommand{\veps}{\varepsilon}
\newcommand{\al}{\alpha}

\newcommand{\gm}{\gamma}
\newcommand{\Gm}{\Gamma}
\newcommand{\tlGm}{\widetilde{\Gm}}

\newcommand{\lm}{\lambda}

\begin{document}

\title[Quasilinear convexity and quasilinear stars]{Quasilinear convexity and quasilinear stars \\[2mm] in the ray space of  \\[2mm] a supertropical  quadratic form}
   \author[Z. Izhakian]{Zur Izhakian}
\address{Institute  of Mathematics,
 University of Aberdeen, AB24 3UE,
Aberdeen,  UK.}
    \email{zzur@abdn.ac.uk}
\author[M. Knebusch]{Manfred Knebusch}
\address{Department of Mathematics,
NWF-I Mathematik, Universit\"at Regensburg 93040 Regensburg,
Germany} \email{manfred.knebusch@mathematik.uni-regensburg.de}

%******************************* AMS classification ***********************
\subjclass[2010]{Primary 15A03, 15A09, 15A15, 16Y60; Secondary
14T05, 15A33, 20M18, 51M20}

%******************************* date *************************************
\date{\today}

%******************************* keywords *********************************

\keywords{Tropical algebra, supertropical modules, bilinear forms,
quadratic forms,  quadratic pairs, ray spaces, convex sets, quasilinear sets, Cauchy-Schwartz ratio.}

%******************************* file name *********************************

%\thanks{\noindent \underline{\hskip 3cm } \\ File name: \jobname}

%******************************* abstract *********************************

\begin{abstract} Relying on rays, we search for submodules of a module $V$ over a supertropical semiring  on which a given anisotropic quadratic form is quasilinear.  Rays are  classes of a  certain equivalence relation on $V$, that  carry a notion of convexity, which is consistent with quasilinearity. A criterion for quasilinearity is specified by a Cauchy-Schwartz ratio which paves the way to a  convex geometry on $\Ray(V)$, supported by a ``supertropical trigonometry''.  Employing  a (partial) quasiordering on $\Ray(V)$, this approach allows for producing convex quasilinear sets of rays, as well as paths, containing a given  quasilinear set in a systematic way.  Minimal paths are endowed with a surprisingly rich combinatorial structure, delivered to  the graph determined by pairs of quasilinear rays -- apparently a fundamental object in the theory of supertropical quadratic forms.
\end{abstract}

\maketitle

{ \small \tableofcontents}

\numberwithin{equation}{section}
\section{Introduction}

This paper is a continuation  of \cite{QF1,QF2}, where quadratic forms on a free supertropical module were introduced and classified, as well as their bilinear companions, providing a tropical version of  trigonometry. As  the Cauchy-Schwarz inequality does not always hold on this setting, the so-called CS-ratio plays a major role in this theory. With this CS-ratio, for a suitable equivalence relation, whose equivalence classes are termed  rays, a type of convex geometry on ray-spaces arises. The paper proceeds the study of this geometry.

In sequel to \cite{IzhakianKnebuschRowen2010Linear,QF1,QF2,VR1,UB}, our underlining structure is taken to be a \textbf{supertropical semiring} ~ $R$ (\Qdefref{0.3} and \cite[\S3]{IKR1}); that is  a semiring $R$ where $e:=1+1$ is  idempotent
(i.e., $e+e = e$) and,   for
all~ $x,y\in R$,  $x+y\in\{x,y\}$ whenever $ex \ne ey$, otherwise  $x+y=ey$.
The ideal $eR$ of $R$ is a bipotent semiring (with unit element $e$),  i.e., $u+v$
is either $u$ or $v$, for any $u,v\in eR$. The total
ordering
\begin{equation*}\label{eq:0.5}
u\le v \dss \Leftrightarrow u+v=v
\end{equation*}
of $eR$, together with the map
$x\mapsto ex$,  determines the addition of $R$:
\begin{equation*}\label{eq:0.6}
x+y=\begin{cases} y&\ \text{if}\ ex<ey,\\
x&\ \text{if}\ ex>ey,\\
ey&\ \text{if}\ ex=ey.
\end{cases}
\end{equation*}
Taking  $y=0$,
$ex=0  \Rightarrow x=0$.
The elements of $\tT(R):=R\sm(eR)$ are called \bfem{tangible}, while those of $\tG(R):=(eR)\sm\{0\}$ are called \bfem{ghost} elements. The
zero $0 = e0$ is regarded mainly as a ghost.  $R$
itself is said to be  \bfem{tangible}, if it is generated by $\tT(R)$ as
a semiring, namely  iff $e\tT(R)=\tG(R).$ When
$\tT(R)\ne\emptyset,$ discarding  the ``superfluous'' ghost elements,
$R':=\tT(R)\cup e\tT(R)\cup\{0\}$
is the largest tangible sub-semiring of $R$.

An $R$-module $V$ over a commutative semiring $R$ (with 1) is defined in the familiar way.
 A \bfem{quadratic form} on $V$ is a function
$q: V\to R$ satisfying
\begin{equation}\label{eq:0.1}
q(ax)=a^2q(x)
\end{equation}
for any $a\in R$, $x\in V,$ such that
\begin{equation}\label{eq:0.2}
q(x+y)= q(x)+q(y)+b(x,y)
\end{equation}
for any $x,y\in V$, where $b:V\times V\to R$ is a bilinear form, called a \bfem{companion} of $q$, not necessarily uniquely
determined by $q$.
 The pair $(q,b)$ is called a \bfem{quadratic pair} on $V.$

A quadratic form $q$ with unique companion is called \bfem{rigid}. This is equivalent to  $q(\veps_i)=0$ for all~
$\veps_i$ of a fixed base $\{ \veps_i\pipe i\in I \}$ of $V$, by \Qthmref{3.5}. $q$ is  \bfem{quasilinear}, if
 $b=0$ is one of its  companions, i.e.,
$q(x+y)=q(x)+q(y)$ for all $x,y\in V.$ These are  the ``diagonal''
forms
\begin{equation*}\label{eq:0.8}
q\bigg(\sum_ix_i\veps_i\bigg)=\sum_i q(\veps_i)x_i^2,
\end{equation*}
since  $(\lm+\mu)^2=\lm^2+\mu^2$ for all
$\lm,\mu\in R,$ cf. \Qpropref{0.5}.

Any quadratic form $q$ on a
free $R$-module can be written as a sum
\begin{equation}\label{eq:0.9}
q=q_{QL}+\rho,
\end{equation}
where $q_{QL}$ is a quasilinear (uniquely determined by $q$) and
$\rho$ is rigid (but not unique), called the
\bfem{quasilinear part} of $q$ and a \bfem{rigid complement}
of $q_{\QL}$ in $q$, cf. \Qsecref{4}.

Aiming to detect  on which parts of
the underlying $R$-module $V$ a quadratic form is quasilinear, \cite{QF2} studies the behavior of a quadratic pair $(q,b)$ on varying
  pairs of
non-zero vectors $(x,y)$ in  $V$,
 mostly for $R$ a tangible supersemifield or $eR$ a (bipotent) semifield.
\cite[Theorem~2.7]{QF2} determines
when a quadratic form  is tangible or rigid.

 A  pair $(x,y)$ is called
\textbf{excessive}, if $b(x,y)^2$ exceeds $ q(x)q(y),$ in the sense \cite[Definition ~2.8]{QF2}, in particular when $q(x) = 0$ or $q(y) = 0.$ However, by \cite[Corollary 2.9]{QF2},  either $(x,y)$ is excessive or the restriction $q|Rx+Ry$ of $q$ is quasilinear.
% (then $(x,y)$ is called \textbf{quasilinear} with respect to $q$).
Nevertheless, this dichotomy
does not depend on the companion $b$ of $q,$ although
$b$ takes part in defining excessiveness.

% In \S\ref{sec:II.2} we delve into a kind of ``\textit{tropical trigonometry}''.

The CS-ratio\footnote{``CS'' is an acronym of
``Cauchy-Schwarz''.} of  \textbf{anisotropic} $x,y \in V$, i.e., $q(x)\ne0,$
$q(y)\ne0,$ is defined as
\begin{equation}\label{eq:0.10}
\CS(x,y):=\frac{eb(x,y)^2}{eq(x)q(y)}\in eR.
\end{equation}
It  serves as a tropical analogy to the familiar formula
$ \cos(x,y) = \frac {\langle x,y\rangle} {\| x\| \; \| y\| } $
in euclidian geometry, and leads to a version  of ``\textit{tropical trigonometry}''.
(By squaring a formula there is no loss of information, since for any supertropical semiring  the map $\lm \mapsto \lm^2$  is an injective endomorphism \cite[Proposition 0.5]{QF1}.)
For any anisotropic vector $w$, the function $x\mapsto
\CS(x,w)$ is subadditive \cite[Theorem 3.6]{QF2}.
% This fact turns out to be of central importance in the whole paper.

A CS-ratio
$\CS(x,y)$ can take values larger than $e = 1_{eR}$, which does not happen
in euclidian geometry; thereby features of noneuclidian geometry arise.
These features are closely related to excessiveness, and are of main interest.
When the set $eR$ is densely ordered, the pair
$(x,y)$ is excessive iff $\CS(x,y)>e.$
 When $eR$ is discrete, $(x,y)$ is excessive if either $\CS(x,y)>c_0,$
with $c_0$ the smallest element of $eR$ larger than $e,$  or
$\CS(x,y)=c_0$ and  $q(x)$ or $q(y)$ is
tangible. The pair $(x,y)$ is \textbf{exotic quasilinear}, if $\CS(x,y) = c_0$ and  both $q(x)$ and $q(y)$ are
ghost \cite[Theorems 2.7 and 2.14]{QF2}.  This behavior bears relevance to
problems of an arithmetical nature for classical quadratic forms,
subject to tropicalization.

%In \cite[\S4]{QF2}, we compile tables of the function
%$(\lm,\mu)\mapsto q(\lm x+\mu y)$ on $(R\sm\{0\})^2$ for given
%excessive $x,y\in V\sm\{0\},$ and then study in detail the CS-ratios
%$\CS(x',y')$ of pairs of vectors $(x',y')$ in $Rx+Ry,$ determining
%in Theorem 4.11 precisely when their space is quasilinear or
%excessive.
% %This completes our account of tropical trigonometry in the present paper.
%A first application shows up in~ \S\ref{sec:II.4}, which describes
%``stropicalization'' (short for ``super-tropicalization'') of a
%classical quadratic form. In Proposition \ref{prop:II.4.1} a
%sufficient condition is given for all stropicalizations to be
%quasi-linear.

A projective version of the theory  is obtained  from the
equivalence relation on $V \sm \00$
  whose classes are called \textbf{rays} and defined as   \cite[\S6]{QF2}:
  %\cite[Theorem~6]{QF2} is the projective version of \cite[Corollary 2.9]{QF2}.
  Vectors $x,y$ in $V \sm \00$
belong to the same ray iff $\lm x = \mu y$ for some $\lm, \mu \in R \sm \00$.  ($\lm, \mu$ need not be invertible as in the usual projective equivalence.) When  $x$ and
$y$ are anisotropic, $\CS(x,y)$ depends only on the rays $X,Y$ of $x,y$ and provides   a well defined CS-ratio $\CS(X,Y)$ for
anisotropic rays $X,Y$, i.e., rays $X,Y$ in $V \sm q^{-1}(0).$
In terms of rays,  subadditivity \cite[Theorem 3.6]{QF2} is better described by employing \textbf{intervals}  $[X,Y]$,  determined  by rays $X,Y$. Given an anisotropic ray $Z \in [X,Y]$ and arbitrary $W$,
\cite[Theorem ~7.7]{QF2} compares
$\CS(Z,W)$ with $\CS(X,W)+\CS(Y,W)$ while
\cite[~Theorem~8.8]{QF2} provides the uniqueness of the boundary of
$[X,Y]$.

%Intervals lead in \S\ref{sec:II.6}-\S\ref{sec:II.8} to a convex geometry on the \textbf{ray space} of $V$, where CS-ratios are of good use.
 On the set $\Ray(V)$ of all rays, called the \textbf{ray space} of $V$, a natural notion of convexity appears: a subset $A \subset \Ray(V)$ is \textbf{convex}, if $[X,Y] \subset A$ for any $X,Y \in A$.
 With this notion, given a quadratic pair $(q,b)$, many problems of trigonometrical nature arise, some of which  are addressed in this paper.

 By studying the CS-ratio $\CS(X,Y)$ for $X \in S ,Y \in T$ in given  disjoint subsets $S,T$ of $\Ray(V)$, we obtain in \S\ref{sec:2} rather subtle separation results for the convex hulls of $S$ and ~$T$ (in the obvious sense) from the Subadditivity Theorem \cite[Theorem 3.6]{QF2}, compiled in Corollary ~\ref{cor:2.4}.

We call a \textbf{pair $(X,Y)$ of rays} in $V$ \textbf{quasilinear} (w.r. to $q$), if the restriction
 $q | Rx + Ry$ is quasilinear for any $x \in X$, $y \in Y$, and call a \textbf{subset} $C \subset \Ray(V)$ \textbf{quasilinear}, if any pairs $(X,Y) $ in $C$ are quasilinear.
   It turns our  that the convex hull of $C$  is again quasilinear. In particular, if $C$ is a quasilinear subset of a  convex set $A \subset \Ray(V)$, then the maximal quasilinear subsets  of $A$ containing $C$ %(existing by Zorn's Lemma, $A$)
   are convex. These objects are of central interest in~ \S\ref{sec:3}--\S\ref{sec:5}.

%
% A second topic in the paper is exhibition of large subsets of $\Ray(V)$ over which $q$ behaves in a quasilinear manner. More precisely, a pair $(X,Y)$ of rays is \textbf{quasilinear} (w.r. to $q$), if the restriction
% $q | Rx + Ry$ is quasilinear for any $x \in X$, $y \in Y$.  Then, a subset $C \subset \Ray(V)$ is quasilinear, if all pairs $(X,Y) $ in $C$ are quasilinear. In particular, the convex hull of $C$ (defined in the obvious sense) is again quasilinear. Given a quasilinear subset $C \subset A$, with  convex $A \subset \Ray(V)$, by Zorn's Lemma, $A$ has a maximal quasilinear subsets containing $C$ which are convex. These objects are of central interest in \S\ref{sec:3}--\S\ref{sec:5}.

The \textbf{QL-star} $\QL(X)$ of a ray $X$ (with respect to $q$) is the set of all $Y \in \Ray(V)$ for which the pair $(X,Y)$ is quasilinear, equivalently, the interval $[X,Y]$ is quasilinear. The QL-stars determine the quasilinear behavior of $q$ on the ray space. But, making this explicit, a major difficulty arises  since  a QL-star often is not quasilinear, and in some cases is not even convex, as explained in \S\ref{sec:3} and \S\ref{sec:13}.

In \S\ref{sec:4} we introduce a (partial) quasiordering on $\Ray(V)$, i.e., a transitive and reflexive relation $\preceq_{\QL}$, defined  by
$$ X  \preceq_{\QL} Y \dss\Leftrightarrow \QL(X) \subset \QL(Y).$$
Given a quasilinear  subset $C$ of $\Ray(V)$, this relation is employed to obtain new quasilinear sets containing $C$ in a systematic  way. Namely, suppose $D$  is  a subset of $\Ray(V)$ such that for every $Z \in D $ there is some $X \in C$ with $X \preceq_{\QL} Z$, then the convex hull $C' $ of $C \cup D$ is again a quasilinear set, called a \textbf{QL-enlargement} of $C$. In particular we obtain  a (unique)
\textbf{maximal QL-enlargement} $E(C)$ of $C$ by taking $D$ to be the union of the up-sets
$X^\uparrow = \{ Z \in \Ray(V) \ds | X \preceq_{\QL} Z \} $ with $X$ running through $C$.\footnote{For systematic reasons $X^\uparrow$ is denoted as $\sat_{\QL} $ in \S\ref{sec:4}.}

In \S\ref{sec:5} we study the family $(C_i \ds | i \in I)$ of maximal quasilinear sets in $\Ray(V)$ containing a given quasilinear set $C$. We determine the union $\bigcup_{i \in I} C_i$ and the intersection
$\tlC = \bigcap_{i \in I} C_i$ of these $C_i$, and prove that $C$ itself is maximal  quasilinear (if and) only if $C$ is the intersection of the  QL-stars of all rays in $C$. (One direction is trivial). I turns out that the maximal enlargement $E(C)$ is contained in $\tlC$.

It is a fundamental matter in the theory of supertropical quadratic forms  to study the \textbf{quasilinear graph} $\Gm_{\QL}(V,q)$ of $q: V \to R$, whose vertices and edges are respectively the  rays and the quasilinear pairs of rays, in particular to describe the path components of $\Gm_{\QL}(V,q)$ and to extract information about the diameters. Due to time and space limitation, we leave these topics to a future study, but in \S\ref{sec:14}--\S\ref{sec:18} we provide a preparation for this.

In \S\ref{sec:14} we define \textbf{enlargements} of a given \textbf{QL-path} $(X_0, \dots, X_n)$ (i.e., a path in $\Gm_{\QL}(V,q)$) by employing suitable QL-enlargements of the intervals $[X_i, X_{i+1}]$. We then develop in  \S\ref{sec:15} a procedure for replacing $(X_0, \dots, X_n)$ ba a shorter path from $X_0$ to $X_n$. This provides a rich interplay between QL-paths and the up-sets and down-sets of the quasiordering $\preceq_{\QL}$ on~ $\Ray(V). $

We can retreat to the case that $T= (X_0,\dots, X_n)$ is a \textbf{direct QL-path}, i.e., no pair $(X_i,X_j)$ in~$T$ with $j > i +1 $ is quasilinear,  and then observe in \S\ref{sec:15} that an upset $Y^\uparrow$ can meet~ $T$ in at most two rays $X_p, X_q$, which are then adjacent , $q = p \pm  1$ (Theorem \ref{thm:15.4}).  When this happens, we call $\{ X_p, X_q\} $ a \textbf{twin pair} in $T$, and call $Y$
an \textbf{anchor of} $\{ X_p, X_q\} $. A ray~ $X_r$ in $T$ which does not appear in a twin pair of~ $T$ is called a \textbf{single} of $T$. We regard any ray ~$Y$ with $X_r \in Y^\uparrow$ (i.e., $X_r\preceq_{\QL} Y$) as an \textbf{anchor of} $Y$.

This gives us a sequence $S = (Y_0, \dots Y_m)$, often not unique, which list anchors of all rays in $T$ in an economic way in the sense that  there is a  minimal monotonic surjective map $\pi:[0,n] \twoheadrightarrow [0,m] $ such that some anchor of each $X_i$ is listed in $(Y_0, \dots, Y_{\pi(i)}).$ (Necessarily $\pi(0) = 0$, $\pi(n) =m$.) We call $S$ an \textbf{anchor set of the QL-path $T$}. It remains a widely open problem to find out which sequences $S$ appear as anchor sets of direct QL-paths, although we obtain some relevant information (cf. Theorem \ref{thm:15.9}).

We are better off if $T$ is a \textbf{minimal QL-path}, i.e., a path of shortest length from $X_0$ to ~$X_n$.
Given an anchor set $S$ of $T = (X_0, \dots, X_n)$,  we search in \S\ref{sec:16}--\S\ref{sec:18}  for partitions  of~ $T$ into subpaths (which, of course, are again minimal) which together with their anchors in $S$ obey a simple and transparent combinatoric.  Our main thrust is on subsequences $(X_p, \dots, X_q)$, $q \geq p +2$, of $T$ such that any two adjacent rays $X_{p+i}$  $X_{p+i+1}$ ($0 \leq i < q-p$) are twins in $T$. We call these subsequences \textbf{flocks}. Given an anchor set $S$ of $T$, we obtain in \S\ref{sec:18} a modification $T'$ of $T$, determined by $S$ in a unique way,  which is again a QL-path from $X_0$ to~ $X_n$ of length $n$, and has a partition into maximal flocks, isolated twin pairs (i.e., twin pairs not appearing in a flock),and singles (i.e., rays not appearing in a twin pair). This partition of $T'$ and the QL-path $T'$ itself are uniquely determined by the anchor set $S$ of $T$. Then, surprisingly, it turns out that  all this procedure does not depend on the choice of $S$, and so is determined by ~$T$ alone.

Here  ample space is left for  further study. In particular, it remains a mystery which sequences of maximal flocks, isolated twin pairs, and singles show up, while  running through all minimal paths  from $X_0$ to $X_n$.

\begin{notation}\label{notation:0.1}
%Let $\mathbb N=\{1,2,3,\dots\}$, $\mathbb N_0=\mathbb N\cup\{0\}.$
$R^*$ denotes the group of units of a semiring $R.$ In a supertropical semiring~$R$
\begin{enumerate} \dispace
\item[$\bullet$] $\tT(R):=R\sm eR$ is the set of tangible
elements $\ne 0$,

 \item[$\bullet$] $\tG(R):=eR\sm \{0\} $ is the
set of ghost elements $\ne 0,$

\item[$\bullet$] $\nu_R$ denotes the ghost map $R \to eR,$ $a\mapsto
ea.$
\end{enumerate}
When it is clear from the context, we write $\tT,$ $\tG,$
$\nu$ for  $\tT(R),$ $\tG(R),$
$\nu_R,$ and $ea=\nu(a)$ for $a \in R.$ $a \leq_\nu b$  (resp. $a <_\nu b$) stands for $ea  \leq eb$ (resp. $ea  < eb$), the
$\nu$-equivalence  $a \nucong b$  means that $ea  = eb$.
\end{notation}

\section{Convex sets in the ray space}\label{sec:1}
We assume that $V$ is an $R$-module over a supertropical semiring $R$ whose ghost ideal $eR$ is a (bipotent) semifield. We compile some facts about convex sets and intervals in $\Ray(V)$ without yet involving a quadratic form.

\begin{defn}\label{def:1.1} $ $
\begin{enumerate}\ealph
  \item  A subset $M$ of $\Ray(V)$ is \textbf{convex} (in $\Ray (V)$), if for any two rays $X, Y \in M$ the closed interval $[X, Y]$ is contained in $M$.

%\item
% The \textbf{convex hull}, denoted by $\conv (S)$, of a nonempty set $S \subset \Ray (V)$ ie the smallest  convex subset of $\Ray (V)$ containing $S$ (which obviously exists). When  $S = \{ X_1, \dots, X_n \}$ is finite, we write
%$ \conv (S) = \conv (X_1, \dots, X_n) $,
%for short.

\item The smallest  convex subset of $\Ray (V)$ containing a nonempty set $S \subset \Ray (V)$  (which obviously exists) is called the \textbf{convex hull} of $S$ (in $\Ray( V)$) and is denoted by $\conv (S)$.  When  $S = \{ X_1, \dots, X_n \}$ is finite, we write
$ \conv (S) = \conv (X_1, \dots, X_n) $,
for short.

\end{enumerate}
\end{defn}
\begin{examp}\label{exmp:1.2}
For any rays $X, Y$ in $\Ray(V)$ all the intervals $] X, Y [ \, , ] X, Y ], [ X, Y [ \, , [ X, Y ]$ (cf. \cite[\S6]{QF2}) are convex sets \cite[Proposition 8.1]{QF2}. Clearly $[X, Y] = \conv (X, Y)$.
\end{examp}

\begin{prop}\label{prop:1.3}

$ $ \begin{enumerate} \ealph
      \item
     If $U_1, \dots, U_n$ are ray-closed subsets of $V \setminus \{ 0 \}$, i.e., unions of full rays, then the set $U_1 + \dots + U_n$ is again ray-closed in $V$, consisting of all rays $\ray_V (\lm_1 u_1 + \dots + \lm_n u_n)$ with $u_i \in U_i$, $\lm_i \in R \setminus \{ 0 \}$. In particular, for any rays $X_1, \dots, X_n$ in $V$ the set $X_1 + \dots + X_n$ is ray-closed in $V$.
\item The convex hull of a finite set of rays $\{ X_1, \dots, X_n \}$ has the disjoint decomposition
\[ \conv (X_1, \dots, X_n) = \bigcup\limits_{i_1 < \dots < i_r} \Ray (X_{i_1} + \dots + X_{i_r}) \]
with $r \leq n$ and  $1 \leq i_1 < \dots < i_r \leq n$.
\end{enumerate}

\end{prop}
\begin{proof}
The claims follow from the description of intervals in \cite[Scholium 7.6]{QF2} by an easy induction. \end{proof}

\begin{notation}\label{notat:1.4}

We denote by  $\tlA$ the subset of  $V \setminus \{ 0 \}$, obtained as the union of all rays in a subset $A \subset \Ray (V)$. In other terms, $\tlA$ is the unique ray-closed subset of $V \setminus \{ 0 \}$ with $\ray (\tlA) = A$.
\end{notation}

\noi
As the convex hull of a subset $A$ of $\Ray (V)$ is the union of all sets $\conv (X_1, \dots, X_r)$ with $r \in \mathbb{N}$, $X_1, \dots, X_r \in A$,   we derive  the following from Proposition~\ref{prop:1.3}.(b).

\begin{cor}\label{cor:1.5} Assume that $A_1, \dots, A_n$ are convex subsets of $\Ray (V)$. Let $C$ denote the convex hull of $A_1 \cup \dots \cup A_n$.
\begin{enumerate}
  \ealph
\item $C$ is the union of all sets $\conv (X_1, \dots, X_n)$ with $X_i \in A_i$,  $1 \leq i \leq n$.
\item $\tlC$ is the union of all sets $\tlA_{i_1} + \dots + \tlA_{i_r}$ with $r \leq n$, $1 \leq i_1 < \dots < i_r \leq n$.
\end{enumerate}
\end{cor}
\begin{proof} (a): Given $X_1, \dots, X_n$, $Y_1, \dots, Y_n$ with $X_i, Y_i \in A_i$ $(1 \leq i \leq n)$ we have
$$ \begin{array}{lllll}
  \conv (\conv (X_1, \dots, X_n), \conv (Y_1, \dots, Y_n)) &  = \conv (X_1, \dots, X_n, Y_1, \dots, Y_n) \\
& = \conv ([X_1, Y_1] \cup \cdots \cup [X_n, Y_n]) \\
& \subset \conv (A_1 \cup \cdots \cup A_n),  \end{array}
$$
since $[X_i, Y_i]$ is the convex hull of $\{ X_i, Y_i \}$. % This gives the first claim.

\noindent
(b): By  Proposition \ref{prop:1.3}.(b)  for $X_i \in A_i$ $(1 \leq i \leq n)$ we have
$$\conv (X_1, \dots, X_n) = \bigcup\limits_{1 \leq i_1 < \dots < i_r \leq n} \Ray (X_{i_1} + \dots + X_{i_r})  \subset \bigcup\limits_{i_1 < \dots < i_r} \Ray (\tlA_{i_1} + \dots + \tlA_{i_r}).$$
This together with part (a) implies  the second claim.
\end{proof}

Alternatively,  Proposition \ref{prop:1.3} and  Corollary \ref{cor:1.5} can be derived from the following observation, which deserves independent interest.

\begin{rem}\label{rem:1.6}
The convex subsets $A$ of $\Ray (V)$ correspond uniquely to the ray-closed submodules $W$ of $V$ via
\[ W = \tlA \cup \{ 0 \} , \quad A = \Ray (W). \]
\end{rem}
\begin{proof}

 This is evident from the fact that for any $X, Y \in \Ray (V)$ the interval $[X, Y]$ is the set of all rays in the module $X_0 + Y_0 = (X + Y) \cup X \cup Y \cup \{ 0 \}$, cf. \cite[Remark~6.4]{QF2}.
\end{proof}

\section{CS-ratios for anisotropic rays}\label{sec:2}

We assume that $V$ is a module over a supertropical semiring $R$, whose ghost ideal $e R$ is a  (bipotent) \textbf{nontrivial semifield}\footnote{The term   ``nontrivial'' means that $\tG \ne \{ e \}$.}, and that a quadratic pair $(q, b)$ on $V$ is given  with $q$ \textbf{anisotropic}. In this situation, the CS-ratio $\CS (X, Y)$ is well defined for any two rays $X, Y$ in the ray space $\Ray (V)$, such that
\begin{equation}\label{eq:2.1}
  \CS (X, Y) = \CS (x, y) =  e \frac{b (x, y)^2}{ q (x) q(y)} \end{equation}
for any $x \in X, y \in Y$, cf. \cite{QF2}.

\medskip\noi
In what  follows we exploit a major result from \cite{QF2} on such CS-ratios, that is \cite[Theorem~7.9]{QF2}, first from part~(a)  there, and later from   parts~(b) and (c).

\begin{thm}\label{thm:2.1} Given  $\gm \in e R$, let $S$ and $T$ be subsets of the ray space $\Ray (V)$ with $\CS (X, Y) \leq \gm$ for all $X \in S$ and $Y \in T$. Then $\CS (Z, W) \leq \gm$ for all $Z$ and  $W$ in the convex hulls of $S$ and $T$ respectively.
\end{thm}
\begin{proof} a): We first prove that $\CS (Z, Y) \leq \gm$ for $Z \in \conv (S)$, $Y \in T$. By Proposition~ \ref{prop:1.3}.(b)  we have rays $X_1, \dots, X_n \in S$ and vectors $x_i \in X_i$ $(1 \leq i \leq n)$ such that
\begin{equation}\label{eq:2.2}
 Z = \ray (x_1 + \dots + x_n) . \end{equation}
By \cite[Theorem 7.9.a]{QF2} there exist elements $\al_1, \dots, \al_n \in \tG$ with $\al_1 + \dots +\al_n = e$ such that
\begin{equation}\label{eq:2.3}
 \CS (Z, Y) \leq \sum\limits_{i = 1}^n \al_i \CS (X_i, Y), \end{equation}
and so
\[ \CS (Z, Y) \leq \sum\limits_{i = 1}^n \al_i \gm = \gm. \]

\noindent
b): Similarly, given a ray $Z \in \conv (S)$, we conclude that $\CS (Z, W) \leq \gm$ for every  $W \in \conv (T)$. \end{proof}

\begin{rem}\label{rem:2.2} $ $
\begin{enumerate} \ealph
  \item
    In the same way, we see that if $\CS (X, Y) < \gm$ for all $X \in S$, $Y \in T$, then $\CS (Z, W) < \gm$ for all $Z \in \conv (S)$, $Y \in \conv (T)$. This can also be deduced from Theorem~2.1 in a purely formal way, since every $Z \in \conv (S)$ is contained in the convex hull of some finite subset $S'$ of $S$.

\item  When, in contrary to our present assumption, $q$ is not anisotropic, we restrict $q$ to the submodule
\[ V_{\an} : = \{ x \in V \ds \vert q (x) \ne 0 \} \cup \{ 0 \} \]
of $V$, and obtain the same result as above for subsets $S$ and $T$ in the ray space of~ $V_{\an}$, which coincides with the convex subset $\Ray (V)_{\an}$ of $\Ray (V)$ consisting of all anisotropic rays in $V$. Here it is important to note that the  $R$-modules considered are not necessarily finitely generated, since even if $V$ is finitely generated, most often $V_{an}$ is not.
\end{enumerate}
\end{rem}

The proof of Theorem \ref{thm:2.1} was based on part (a) of \cite[Theorem~7.9]{QF2}. Next we exploit parts (b) and (c) of this theorem, which use the notion of $\nu$-\textit{quasilinearity} of pairs of rays. We briefly recall this notion from \cite{QF2}, assuming that $e R$ is a nontrivial semifield.

 A pair of vectors $x, y \in V$ is called $\nu$-\textbf{quasilinear}, if
\[ q (x+y) \cong_{\nu} q (x) + q (y), \]
or equivalently, if the pair $(x, y)$ is quasilinear with respect to the quadratic form $e q$. A {pair of rays} $X, Y$ in $V$ is said to be  $\nu$-\textbf{quasilinear}, if $(x, y)$ is $\nu$-quasilinear for all $x \in X$, $y \in Y$.
When the rays $X, Y$ are anisotropic -- a standard assumption in this section --  the pair $(X, Y)$ turns out to be $\nu$-quasilinear iff $\CS (X, Y) < c$ for every $c > e$. In other terms, is the semifield $e R$ is dense,   $\CS (X, Y) \leq e$, while, if $e R$ discrete,   $\CS (X, Y) \leq c_0$ where $c_0$ is the smallest element of $e R$ that is bigger than $e$, cf. \cite[Definition~7.3]{QF2}. % \com{CHECK THIS}

\begin{thm}\label{thm:2.3}
 Assume again that $e R$ is a nontrivial semifield and that $q$ is anisotropic. Let~ $S$ and $T$ be subsets of $\Ray (V)$ such that for any two rays $X \in S$, $Y \in T$ the pair $(X, Y)$ is $\nu$-quasilinear and $\CS (X, Y)$ is contained in a given convex subset $\Gamma$ of $e R$.
\begin{enumerate} \ealph
\item Then every pair $(Z, W)$ with $Z \in \conv (S)$, $W \in \conv (T)$ is $\nu$-quasilinear, and $\CS (Z, W) \in \Gamma$.
\item If in addition all $\CS$-ratios $\CS (X, Y)$ with $X \in S$, $Y \in T$ are contained in a fixed square class $\Delta$ of $ e R$, cf. \cite[Definition~7.1]{QF1}, then the same holds for all $\CS$-ratios $\CS (Z, W)$ with $Z \in \conv (S)$, $W \in \conv (T)$.
\end{enumerate}
\end{thm}
\begin{proof} By assumption all $\CS$-ratios $\CS (X, Y)$ with $X \in S$, $Y \in T$ are contained in the convex set $\Gamma \cap [0, e]$, if $e R$ is dense,  or  in convex set $\Gamma \cap [0, c_0]$, when $ e R$ is discrete. Replacing ~$\Gamma$ by this smaller convex set, without loss of generality we may assume that $\Gamma \subset [0, e]$, respectively $\Gamma \subset [0, c_0]$.

To prove part (a) we follow the proof of Theorem \ref{thm:2.1}. We first pick rays $Z \in \conv (S)$ and $Y \in T$, for which we have rays $X_1, \dots, X_n \in S$ and vectors $x_i \in X_i$ such that $Z = \ray (x_1 + \dots + x_n)$. By \cite[Theorem~7.9.c]{QF2} there exist  $\al_1, \dots, \al_n$ in ~$\tG$ with $\al_1 + \dots + \al_n = e$ such that
\begin{equation}\label{eq:2.4}
 \CS (Z, Y) = \sum\limits_{i = 1}^n \al_i \CS (X, Y). \end{equation}
As the pairs $(X_i, Y)$ are $\nu$-quasilinear and all values $\CS (X_i, Y)$ are in ~$\Gamma$, also $\CS (Z, Y) \in \Gamma$, and so $(Z, Y)$ is $\nu$-quasilinear. Applying the same argument to a fixed $Z \in \conv (S)$ and varying $W \in \conv (T)$, we  conclude that $\CS (Z, W) \in \Gamma$ for all $Z \in \conv (S)$, $W \in \conv (T)$.

To prove part (b)  we recall from \cite[Remark~6.8]{QF2} that for any ray $X$ in $V$ the set $e q (X)$ is a square class of the semifield $e R$. If $X$ and $Y$ are rays in~$V$ and $x \in X$, $y \in Y$, then  Formula \eqref{eq:2.1}
%\[ \CS (X, Y) = \frac{ e b (x, y)^2}{ e q (x) q (y)} \]
tells us that $\CS (X, Y) \in eq (X) \cdot eq (Y)$. Assume that all $\CS$-ratios $\CS (X, Y)$ with $X \in S$, $Y \in T$ lie in a fixed square class $\Delta$ of $e R$. Given $Y \in T$, we conclude that
\begin{equation}\label{eq:2.5}
 eq (X) = \Delta \cdot q (Y) \end{equation}
for all $X \in S$. Then  \cite[Theorem~7.9.b]{QF2} applies for a given $Z \in  \conv (S)$, and instead of \eqref{eq:2.4} we obtain the  relation
\begin{equation}\label{eq:2.6}
 \CS (Z, Y) = \sum\limits_{i = 1}^n \al_i^2 \CS (X_i, Y)\end{equation}
with $X_i \in S$, $\al_i \in e R$, $\al_1 + \dots + \al_n = e$.
Namely, $\CS (Z, Y)$ is the maximum over the elements $\al_i^2 \CS (X_i, Y)$ in $e R$, whence $\CS (Z, Y) \in \Delta$. Since this holds for all $Z \in \conv (S)$, $Y \in T$,  repeating the argument, we obtain that $\CS (Z, W) \in \Delta$ for all $Z \in \conv (S)$, $W \in \conv (T)$, as claimed.
\end{proof}

We draw the following consequences from Theorems \ref{thm:2.1} and \ref{thm:2.3} about disjointness of convex hulls in the ray space.

\begin{cor}\label{cor:2.4}  $ $
\begin{enumerate}\ealph
  \item
 Assume that the pair $(q, b)$ is balanced, i.e., $b(x, x) = e q (x)$ for all $x \in V$, cf. \cite[\S1]{QF1}\footnote{Recall from \cite[\S 1]{QF1} that every quadratic form on $V$ has a balanced companion.}. Let $S$ and $T$ be subsets of $\Ray (V)$ with $\CS (X, Y) < e$ for all $X \in S$, $Y \in T$. Then the convex hulls of $S$ and $T$ are disjoint.

 \item  Assume that $e R$ is discrete. Let $S$ and $T$ be subsets of $\Ray (V)$ with $\CS (X, Y) = c_0$ for all $X \in S$, $Y \in T$. Then again $\conv (S) \cap \conv (T) = \emptyset$.
\item  Let $(q, b)$ be balanced. Assume that all $\CS$-ratios $\CS (X, Y)$ with $X \in S$, $Y \in T$ are contained in a fixed square class $\Delta \ne e R^2$, and furthermore  that every pair $(X, Y)$ with $X \in S$, $Y \in T$ is $\nu$-quasilinear. Then again $\conv (S) \cap \conv (T) = \emptyset$.
\end{enumerate}
\end{cor}
\begin{proof} a): It follows from Theorem \ref{thm:2.3}.(a), applied with $\Gamma = [0, e[$, that $\CS (Z, W) < e$ for all $Z \in \conv (S)$, $W \in \conv (T)$. But $\CS (Z, Z) = e$ for all $Z \in \Ray (V)$, and thus  $\conv (S) \cap \conv (T) = \emptyset$.
Alternatively, we obtain this result from Remark \ref{rem:2.2}.(a), applied with $\gm = e$.\pSkip
b): We obtain from Theorem \ref{thm:2.3}.(a) that $\CS (Z, W) = c_0$ for all $Z \in \conv (S)$, $W \in \conv (T)$, while $\CS (Z, Z) \leq e$ for every $Z \in \Ray (V)$ (cf. \cite[Eq. (1.7)]{QF1}).\pSkip
c): We know by Theorem \ref{thm:2.3}.(b) that $\CS (Z, W) \in \Delta$ for any $Z \in \conv (S)$, $W \in \conv (T)$, while $\CS (Z, Z) = e$ for every $Z \in \Ray (V)$. \end{proof}

\section{Quasilinear sets and QL-stars}\label{sec:3}

We dismiss  the assumption in  \S\ref{sec:2} that $q$ is anisotropic, and first only assume that $V$ is an  $R$-module
over a supertropical semiring
such that the pair $(R, V)$ is \textbf{ray-admissible}, i.e., for any $\lm, \mu \in R$ and any $v \in V$
\begin{equation}\label{eq:3.1}
 \lm \ne 0, \mu \ne 0 \dss \Dir \lm \mu \ne 0, \end{equation}
\begin{equation}\label{eq:3.2}
 \lm \ne 0, v \ne 0 \dss \Dir \lm v \ne 0, \end{equation}
so that the definition of rays in $V$ makes sense (cf. \cite[\S6]{QF2}). We briefly say  that the \textbf{$R$-module~ $V$ is ray-admissible}. (Note that \eqref{eq:3.1} means that the semiring $R$ has no zero-divisors.)

We will exploit the following result, proved in {\cite[Proposition 1.20]{QF1}, which   holds for any module $V$ over any semiring $R$.

\begin{thm}\label{thm:3.1} Assume that $q : V \to R$ and $b : V \to R$ are a quadratic and a symmetric bilinear form on $V$,   that $(x_i \ds \vert i \in I)$ is a family of vectors in $V$, and that $b$ accompanies $q$ on the set $S : = \bigcup\limits_{i \in I} R x_i$, i.e.,
\[ q (s+t) = q (s) + q(t) + b (s, t) \]
for any $s, t \in S$. Then $b$ accompanies $q$ on the submodule $\sum\limits_{i \in I} R x_i$ of $V$, generated by $S$.
\end{thm}

\begin{defn}\label{def:3.2}Given a quadratic form $q$ on $V$, we call a (nonempty) subset $U$ of $\Ray (V)$ \textbf{quasilinear} (w.r. to $q$) if for any two rays $X, Y \in U$ the pair $(X, Y)$ is quasilinear, i.e.,
\[ q (x+y) = q (x) + q (y) \]
for any two vectors $x \in X$, $y \in Y$.
\end{defn}

\begin{thm}\label{thm:3.3}Assume that $q : V \to R$ is any quadratic form on the  (ray-admissible) $R$-module $V$ and that $S$ is a quasilinear subset of $\Ray (V)$ (w.r. to $q$). Then the convex hull $\conv (S)$ is again quasilinear (w.r. to $q$).
\end{thm}  \begin{proof}Apply Theorem 3.1 with $b = 0$. \end{proof}

\begin{rem}\label{rem:3.4}
 We know that $\conv (X, Y) = [X, Y]$ for any two rays $X, Y$ in $V$. Theorem \ref{thm:3.3} tells us that, if the pair $(X, Y)$ is quasilinear, then the set $[X, Y]$ is quasilinear (in the sense of Definition \ref{def:3.2}). On the other hand, assuming that $R$ is supertropical with $eR$ a nontrivial semifield, in \cite[\S8]{QF2},   a closed interval $[X, Y]$ is \textbf{defined}  to be quasilinear, if the pair $(X, Y)$ is quasilinear\footnote{This makes sense since for a closed interval $I = [X, Y]$ the boundary rays $X, Y$ are uniquely determined by $I$ up to permutation \cite[Theorem~7.9]{QF2}}. But, as a consequence of Theorem~\ref{thm:3.3}, these notions of quasilinearity coincide  for closed intervals.
\end{rem}

We are ready for a key definition in this paper,  assuming that the $R$-module $V$  is ray-admissible, where  $R$ is any supertropical semiring.

\begin{defn}\label{def:3.5} The QL-\textbf{star} $\QL (X)$ of a ray $X$ in $V$ is the set of all $Y \in \Ray (V)$ such that the  pair $(X, Y)$ is quasilinear, equivalently, that the closed interval $[X, Y]$ is quasilinear.
\end{defn}

It is evident from Definition \ref{def:3.5}, that
\begin{equation}\label{eq:3.3}
 X \in \QL (X) \end{equation}
and, for any two rays $X, Y$ in $V$,
\begin{equation}\label{eq:3.4}
 X \in \QL (Y) \Iff Y \in \QL (X).\end{equation}
Often $\QL (X)$ is \textbf{not a quasilinear convex set} (cf. Remark~\ref{rem:3.7} below), and in rare cases $\QL (X)$ is not convex at all (cf. \S\ref{sec:13} below). But we have the following useful fact.

\begin{thm}\label{thm:3.6} Assume that $X, Y, Z$ are rays in $V$ with $Y \in \QL (X)$, $Z \in \QL (X)$ and  that the pair $(Y, Z)$ is quasilinear. Then $[Y, Z] \subset \QL (X)$, and moreover
\begin{equation}\label{eq:3,5}
 \conv (X, Y, Z) \subset \QL (X) \cap \QL (Y) \cap \QL (Z).\end{equation}
\end{thm}
\begin{proof} All three intervals $[X, Y], [X, Z], [Y, Z]$ are quasilinear, and we conclude by  Theorem~\ref{thm:3.3} that the set $S  = \conv (X, Y, Z)$ is also quasilinear. It follows that for any $W \in S$ the interval $[X, W]$ is contained in $S$, whence $S \subset \QL (X)$. Furthermore $[Y, Z] \subset \QL(X)$, since ~$S$ is quasilinear. By symmetry also $S \subset \QL (Y)$ and $S \subset \QL (Z)$, and so $S$ is contained in the intersection of the QL-stars of $X, Y, Z$.
\end{proof}
\begin{rem}\label{rem:3.7}
Condition~\eqref{eq:3.5} is weaker than the condition that all three pairs $(X, Y)$, $(X, Z)$, $(Y, Z)$ are quasilinear, since for three rays $X, Y, Z$ in $V$ with $\QL (X) \cap \QL (Y) \cap \QL (Z) \ne \emptyset$ it  may happen that $\conv (X, Y, Z)$ is not quasilinear.   Take for example a free module $V$ of rank~4 with base $(\veps_i \ds \vert i \leq i \leq 4)$ over, say, a nontrivial tangible supersemifield~ $R$, and consider the quadratic form
\[ \left[ \begin{array}{cccc}
1 & \gm & \gm & 0 \\
  & 1      & \gm & 0 \\
  &        & 1      & 0 \\
  &        &        & 1
\end{array} \right] \]
for some $\gm >_{\nu} 1$. Then the four basic rays $X_i = \ray (\veps_i)$ have the property that $X_4 \in \QL (X_1) \cap \QL (X_2) \cap \QL (X_3)$, but $\conv (X_1, X_2, X_3)$ is not quasilinear.
Note also that $\QL (X_4) = \Ray (V)$, so $\QL (X_4)$ is convex but, of course, not quasilinear.
\end{rem}

We next analyze the situation that $\QL (X) \subset \QL (X')$ for given rays $X, X'$ in $V$. As a
 preparation,  we study intersections of QL-stars. For a  \textit{nonempty} subset $S$ of $\Ray (V)$ we set
\begin{equation}\label{eq:3.5}
 \QL (S) : = \bigcap\limits_{X \in S} \QL (X).\end{equation}
We read off from \eqref{eq:3.4} that
%\begin{equation}\label{eq:3.6}
 %\QL (S) = \{ Y \in \Ray (V) \ds \vert X \in \QL (Y) \}.\end{equation}
%Furthermore,
if $T$ is a second nonempty subset of $\Ray (V)$, then
\begin{equation}\label{eq:3.7}
S \subset \QL (T) \Iff T \subset \QL (S).\end{equation}
It may happen that $\QL (S)$ is empty, but otherwise the following holds.

\begin{lem}\label{lem:3.8} Assume that $\emptyset \neq  S \subset \Ray (V)$ and $\QL (S) \ne \emptyset$.
\begin{enumerate} \ealph
  \item Then
\begin{equation}\label{eq:3.8}
 S \subset \QL (\QL (S)). \end{equation}

  \item  If $S$ is quasilinear then
\begin{equation}\label{eq:3.9}
 S \subset \QL (S). \end{equation}

\end{enumerate}

\end{lem}

\begin{proof} (a): Given $X \in S$ we have $X \in \QL (Y)$ for every $Y \in \QL (S)$ (cf. \eqref{eq:3.7}), and so $S \subset \QL (\QL (S))$.\pSkip
(b): $S$ is contained in $\QL (X)$ for any $X \in S$, since $S$ is quasilinear, and so $S \subset \QL (S)$.
\end{proof}

\begin{lem}\label{lem:3.9}
 If $S$ and $T$ are subsets of $\Ray (V)$ with  both $\QL (S)$ and $\QL (T)$ nonempty, then
\[ \QL (S) \subset \QL (T) \dss \Iff T \subset \QL (\QL (S)). \]
\end{lem}
\begin{proof} $\QL (S) \subset \QL (T) \Iff \forall Y \in T : \QL (S) \subset \QL (Y)
\Iff \forall Y \in T : Y \in  \QL (\QL (S))$ (cf. ~\eqref{eq:3.7}) $
\Iff T \subset \QL (\QL (S)).$ \end{proof}

\begin{prop}\label{prop:3.10}
If $S \subset \Ray (V)$ and $\QL (S) \ne \emptyset$, then $\QL (\QL (S))$ is the biggest set $T \supset S$ in $\Ray (V)$ such that $\QL (S) = \QL (T)$.
\end{prop}
\begin{proof} Let $T \supset S$. Then, of course, $\QL (T) \subset \QL (S)$ and thus $\QL (T) = \QL (S)$ iff $\QL (T) \supset \QL (S)$, which by Lemma \ref{lem:3.9}  occurs iff $T \subset \QL (\QL (S))$.
\end{proof}

Supported by this proposition, we call the set $\QL (\QL (S))$ the QL-\textbf{saturum of} $S$, provided that $\QL (S) \ne \emptyset$, and  write
\begin{equation}\label{eq:3.10}
 \sat_{\QL} (S) : = % \sat_{\QL} (\{ X \}) =
 \QL (\QL (S)). \end{equation}
We now  turn to handle conveniently inclusion relations between QL-stars, focusing on the case $S = \{ X \}$ for a single ray $X$ in $V$ in which  $X \in \QL (S)$, and thus certainly $\QL (S) \ne \emptyset$.
Writing
%\begin{defn}\label{def:3.11}
% We call
\begin{equation}\label{eq:11b}
\sat_{\QL} (X) : = \sat_{\QL} (\{ X \}) = \QL (\QL (X)),
\end{equation}
%the QL-\textbf{saturum} of the ray $X$.\end{defn}
we obtain
\begin{thm}\label{thm:3.12} If $X$ and $X'$ are rays in $V$, then $\QL (X) \subset \QL (X')$ iff $X' \in \sat_{\QL} (X)$.\end{thm}

\begin{proof}

 A special case of Lemma \ref{lem:3.9}.
\end{proof}

\section{QL-enlargements}\label{sec:4}

We introduce a (partial) quasiordering $\preceq_{\QL}$  on $\Ray (V)$, i.e., a reflexive and transitive binary relation but not necessarily  antisymmetric. For $X_1, X_2 \in \Ray (V)$ we set
\begin{equation}\label{eq:4.1}
 X_1 \preceq_{\QL} X_2  \Iff \QL (X_1) \subset \QL (X_2), \end{equation}
which induces the equivalence relation:
\begin{equation}\label{eq:4.2}
 X_1 \sim_{\QL} X_2 : \Iff X_1 \preceq_{\QL} X_2 \; \mbox{and} \; X_2 \preceq_{\QL} X_1 \Iff \QL (X_1) = \QL (X_2). \end{equation}
Then, by Theorem \ref{thm:3.12} we obtain
\begin{equation}\label{eq:4.3}
 X_1 \preceq_{\QL} X_2 \Iff X_2 \in \sat_{\QL} (X_1), \end{equation}
and so
\[ X_1 \sim_{\QL} X_2 \Iff X_2 \in \sat_{\QL} (X_1), X_1 \in \sat_{\QL} (X_2). \]

Using Lemma \ref{lem:3.9}, we arrive at a third description of the quasiordering $\preceq_{\QL}$, namely
\begin{equation}\label{eq:4.4}
 X_1 \preceq_{\QL} X_2 \Iff \sat_{\QL} (X_2) \subset \sat_{\QL} (X_1), \end{equation}
and so \begin{equation}\label{eq:4.5}
X_1 \sim_{\QL} X_2 \Iff \sat_{\QL} (X_1) = \sat_{\QL} (X_2). \end{equation}

\bigskip\noi
\textit{Caution.} In general it is \textit{not true that} $\sat_{\QL} (X)$ is the set of all $Y \in \Ray (V)$ with $X \sim_{\QL} Y$.

\begin{rem}\label{rem:3.12}
Given a convex subset $A$ of $\Ray(V)$ we can establish the present theory for quasilinear subsets of $A$ instead of $\Ray(V)$ by defining for $X \in A$ the \textbf{relative QL-star}
\begin{equation}\label{eq:3.12}
\QL^A(X) = \QL(X) \cap A
\end{equation}
and defining on $A$ the quasioredering
\begin{equation}\label{eq:3.14}
X \preceq_{\QL,A} Y \dss \Leftrightarrow = \QL^A(X) \subset \QL^A(Y)
\end{equation}
with associated equivalence relation
\begin{equation}\label{eq:3.14}
X \sim _{\QL,A} Y \dss \Leftrightarrow = \QL^A(X) = \QL^A(Y).
\end{equation}
But this is nothing really new, since the $R$-module $V$ can be replaced by the ray-closed $R$-submodule $W = \tlA \cup \{ 0 \}$ defined in \S\ref{sec:1}.
\end{rem}
This illustrates  that it can be useful to work with quadratic pairs on a ray-admissible $R$-module instead of, say, just a free $R$-module.

\pSkip

The relation $\preceq_{\QL}$ allows to produce new quasilinear convex sets in $\Ray (V)$ from old ones. (All this holds for $R$ any supertropical semiring.)

\begin{lem}\label{lem:4.2} Let $X_1, X_2, X'_1, X'_2$ be rays in $V$ such that $X_1 \preceq_{\QL} X'_1$, $X_2 \preceq_{\QL} X'_2$, and the pair $(X_1, X_2)$ is quasilinear. Then $\conv (X_1, X_2, X'_1, X'_2)$ is a quasilinear convex set. \end{lem}

\begin{proof} Since $X_2 \in \QL (X_1)$ and $\QL (X_1) \subset \QL (X'_1)$, also $X_2 \in \QL (X'_1)$, i.e., $(X'_1, X_2)$ is quasilinear. Similarly, the pair $(X_1, X'_2)$ is quasilinear. From $X'_1 \in \QL (X_2)$ we infer that $X'_1 \in \QL (X'_2)$, and so $(X'_1, X'_2)$ is quasilinear. Now Theorem \ref{thm:3.3} implies that $\conv (X_1, X_2, X'_1, X'_2)$ is quasilinear. \end{proof}

\begin{thm}\label{thm:4.3} Given a quasilinear convex subset $C$ of $\Ray (V)$, assume that $D$ is a subset of $\Ray (V)$ such that for every $Z \in D$ there exists some $X \in C$ with $X \preceq_{\QL} Z$. Then $\conv (D)$ is quasilinear. \end{thm}

\begin{proof}  Let $Z, W \in D$ be given, choose $X, Y \in C$ with $X \preceq_{\QL} Z$, $Y \preceq_{\QL} W$. Then $\conv (X, Y, Z, W)$ is a quasilinear convex set by Lemma \ref{lem:4.2}. Thus, the pair $(Z, W)$ is quasilinear and hence $\conv (D)$ is quasilinear, by Theorem~\ref{thm:3.3}. \end{proof}

Assume again that $C$ is a quasilinear convex subset of $\Ray (V)$ and $D$ is a subset of $\Ray (V)$ such that for every $Z \in D$ there is some $X \in C$ with $X \preceq_{\QL} Z$. Then we infer from Theorem~4.3 that
\begin{equation}\label{eq:4.6}
 C' : = \conv (D \cup C) \end{equation}
is \textit{a quasilinear convex set containing} $C$.

\begin{defn}\label{def:4.4} A set $C'$ of the form \eqref{eq:4.6} is called an \textbf{enlargement} of the quasilinear convex set $C$.
We also say that the \textbf{pair} $C \subset C'$ is a \textbf{quasilinear enlargement} (or $\QL$-\textbf{enlargement},  for short).
The  set  $D$ in \eqref{eq:4.6} is said to be a \textbf{mother set} of $C'$ (over $C$). Note that then $D \setminus C$ is also a mother set of $C'$. A \textbf{disjoint mother set} of $C'$ over $C$ is  a mother set $D_1$ of $C'$ with $D_1 \cap C = \emptyset$.   \end{defn}

The following is now obvious.

\begin{rem}\label{rem:4.5}
 Let  $C$ be a quasilinear convex (non-empty) subset of $\Ray (V)$.
\begin{itemize}
\item[a)] The maximal mother set occurring for any enlargement of $C$ is
\[ D_{\infty} : = \{ Z \in \Ray (V) \ds\vert \exists \, X \in C : X \preceq_{\QL} Z \}, \]
in other terms (cf. (3.9)),
\begin{equation}\label{eq:4.7}
 D_{\infty} = \bigcup\limits_{X \in C} \sat_{\QL} (X). \end{equation}
\item[b)] The subsets of $D_{\infty}$ are precisely all mother sets of all enlargements of $C$, and
\begin{equation}\label{eq:4.8}
 E(C) : = \conv (D_{\infty}) \end{equation}
is the unique maximal enlargement of $C$.
\end{itemize}\end{rem}

Not every convex set $C'$ with $C \subset C' \subset E (C)$ is an enlargement of $C$. But, this is true when $C$ is a singleton $\{ X_0 \}$, as a consequence of the next theorem and Corollary~\ref{cor:4.7} below.

\begin{thm}\label{thm:4.6} The $\QL$-saturum $\sat_{\QL} (X_0)$ of a ray $X_0$ in $V$ is a quasilinear convex subset of $\Ray (V)$.
\end{thm}
\begin{proof} Let $X_1, X_2 \in \sat_{\QL} (X_0)$ and $Z \in [X_1, X_2]$. We verify that $Z \in \sat_{\QL} (X_0)$, which means that $\QL (X_0) \subset \QL (Z)$, cf. \eqref{eq:4.1}. Applying  Theorem~\ref{thm:4.3}   to $D : = \{ X_1, X_2 \}$, we see that $\conv (X_0, X_1, X_2)$ is quasilinear. Given $W \in \QL (X_0)$,  all pairs $(W, X_i)$, $0 \leq i \leq 2$, are quasilinear, since $\QL (X_0) \subset \QL (X_i)$. From Theorem~\ref{thm:3.3} it follows that $\conv (X_0, X_1, X_2, W)$ is quasilinear, and then, that $[W, Z]$ is quasilinear. Thus $W \in \QL (Z)$, which proves that $\QL (X_0) \subset \QL (Z)$, as desired. \end{proof}

The following is now obvious.

\begin{cor}\label{cor:4.7} If  $X$ is a ray in $V$, then
   $\sat_{\QL} (X)$  is the maximal enlargement of the quasilinear convex set $\{ X \}$. The convex sets $C \subset \sat_{\QL} (X)$ with $X \in C$ are precisely all enlargements of $\{ X \}$. \end{cor}

The $\QL$-enlargements $\{ X \} \subset C$ with $C$ convex in $\sat_{\QL} (X)$ may be regarded as the ``atoms'' (or perhaps better ``molecules'') in the set of all enlargements in the ray space $\Ray (V)$. To elaborate  this view  we introduce the concept of ``\textit{amalgamating}'' a family of enlargements.

\begin{prop}\label{prop:4.8}Let $(C_{i0} \subset C_i \ds \vert i \in I)$ be a family of $\QL$-enlargements in $\Ray (V)$. Define
\begin{equation}\label{eq:4.9}
 C_0 : = \conv \bigg( \bigcup\limits_{i \in I} C_{i0} \bigg), \qquad C : = \conv \bigg( \bigcup\limits_{i \in I} C_i \bigg). \end{equation}
Then $C_0 \subset C$ is again a $\QL$-enlargement. If $D_i$ is a mother set of $C_i$ over $C_{i0}$, then $D : = \bigcup\limits_{i \in I} D_i$ is a mother set of $C$ over $C_0$.
\end{prop}

\begin{proof} We have a chain of equalities of convex hulls:
\[ \conv (D \cup C_0) = \conv \bigg( D \cup \bigcup\limits_{i} C_{i0} \bigg) = \conv \bigg( \bigcup\limits_{i} (D_i \cup C_{i0}) \bigg) = \conv \bigg( \bigcup\limits_{i} C_i \bigg) = C. \]
Furthermore, for a given ray $Z \in D$, there exists some $i \in I$ with $Z \in D_i$, and so some ray $X \preceq_{\QL} Z$ with $X \in C_{i0}$. Thus $X \in C_0$. \end{proof}

\begin{defn}\label{def:4.9} $ $
\begin{itemize}
\item[a)] Given a family $(C_{i0} \subset C_i \ds \vert i \in I)$ of $\QL$-enlargements in $\Ray (V)$, we call the enlargement $C_0 \subset C$ obtained in Proposition~\ref{prop:4.8}, cf. \eqref{eq:4.9}, the \textbf{amalgamation} of this family. The amalgamation $C_0 \subset C$ is called \textbf{special}, if $C = \bigcup\limits_{i \in I} C_i$ (instead of $C = \conv \big( \bigcup\limits_i  C_i \big)$), and \textbf{very special}, if in addition $C_0 = \bigcup\limits_{i \in I} C_{0i}$.
    \item[b)] We call a  $\QL$-enlargement $C_0 \subset C$ \textbf{atomic}, if $C_0$ is one point set $\{ X \}$ in $\Ray (V)$. Note that then $C$ is a convex subset of $\sat_{\QL} (X)$ containing $X$ (Corollary \ref{cor:4.7}).
\end{itemize}
\end{defn}

Special amalgamations of families of atomic $\QL$-enlargements arise naturally as follows.

\begin{schol}
  \label{sch:4.10} Let $C$ be a quasilinear convex set in $\Ray (V)$. Choose a family $\{\sat_{\QL} (X_i) \ds \vert i \in~ I \}$ of $\QL$-saturations of rays $X_i \in C$ which covers the set $C$, i.e.,
\[ C \subset \bigcup\limits_{i \in I} \sat_{\QL} (X_i). \]
Let $C_i : = C \cap \sat_{\QL} (X_i)$ and $C_0 : = \conv (X_i \vert i \in I)$. Then $C_0 \subset C$ is a special amalgamation of the family $(\{ X_i \} \subset C_i \ds \vert i \in I)$.
\end{schol}

Families of atomic enlargements are of help to produce  $\QL$-enlargements $C_0 \subset C$ for a quasilinear convex set $C$ with $C_0$ ``small''. Note that for   inclusions $C_0 \subset C_1 \subset C$ of  quasilinear convex sets, where  $C_0 \subset C$ is a $\QL$-enlargement, the inclusion  $C_1 \subset C$ is an $\QL$-enlargement. We  say that the $\QL$-enlargement $C_0 \subset C$ \textbf{encompasses} the $\QL$-enlargement $C_1 \subset C$.

\begin{construction}\label{constr:4.11}

 Given a $\QL$-enlargement $C_1 \subset C$, we aim  for a family of atomic $\QL$-enlargements whose amalgamation encompasses $C_1 \subset C$. We first choose a mother set $D \supset C_1$ of $C$ and  a subset $S \subset C_1$ such that for every $Z \in D$ there is some $X \in S$ with $X \preceq_{\QL} Z$.  Taking a labeling $S = \{ X_i \ds \vert i \in I \}$ of $S$, we define $E_i : = C \cap \sat_{\QL} (X_i)$. Clearly $D \subset \bigcup\limits_{i \in I} E_i$, and so
\[ C = \conv (D) = \conv \bigg( \bigcup\limits_{i \in I} E_i \bigg). \]
Furthermore  $C_0 = \conv (X_i \ds \vert i \in I)$ for $C_0 : = \conv (S) \subset C_1$. Thus $C_0 \subset C$ is the amalgamation of the family $(\{ X_i \} \subset E_i \ds | {i \in I})$, and $C_0 \subset C_1$.
\end{construction}

\section{Maximal quasilinear sets}\label{sec:5}
Given a nonempty quasilinear subset $S$ of $\Ray (V)$, by Zorn's Lemma  there exists a maximal quasilinear set $C' \supset S$ in $\Ray (V)$, because the union of a chain of quasilinear sets obviously is again quasilinear. Since the convex hull of any quasilinear set $S$ is again quasilinear (Theorem~\ref{thm:3.3}), the maximal quasilinear set $C'$ is convex and moreover $C' \supset \conv (S)$. Thus in the search for maximal quasilinear sets containing $S$ we may assume from the beginning that $S$ is convex.% This helps to fix ideas.
However,  many of the formal arguments below remain valid without this convexity assumption.

In what follows $C$ denotes  a fixed nonempty quasilinear subset of $\Ray (V)$ and  $(C_i \ds \vert i \in I)$ denotes the set of all maximal quasilinear sets of $\Ray (V)$ containing $C$.

\begin{thm}\label{thm:5.1}
$\QL (C) = \bigcup\limits_{i \in I} C_i$.
\end{thm}

\begin{proof}
 Since every $C_i$ is quasilinear and $C_i \supset C$, it is obvious that $C_i \subset \QL (C)$. If $X \in \QL(C)$ is given, then for every $Y \in C$ the pair $(X, Y)$ is quasilinear. Thus the set $C \cup \{ X \}$ is quasilinear, directly by Definition~3.2, and so $X \in C_i$ for some $i \in I$. \end{proof}

\begin{cor}\label{cor:5.2} $C$ is maximal quasilinear iff $\QL (C) = C$. \end{cor}

\begin{proof}
 This is the case $\vert I \vert = 1$ of the theorem. \end{proof}

We define
\begin{equation}\label{eq:5.1}
   \tlC : = \bigcap\limits_{i \in I} C_i. \end{equation}
This is the maximal quasilinear set containing $C$, such that $(C_i \ds \vert i \in I)$ as well is the family of all maximal quasilinear sets containing $\tlC$.

%In the following we often write ``ql'' instead of ``quasilinear'', for short.
We denote  by $\Max (C)$ the set of maximal quasilinear subsets of $\Ray (V)$ which contain $C$, assuming tacitly that $C \ne \emptyset$. Thus, in the above notation,
\begin{equation}\label{eq:5.2} \Max (C) : =(C_i \ds \vert i \in I ). \end{equation}
We  have just observed that
\begin{equation}\label{eq:5.3} \Max (C) = \Max (\tlC ). \end{equation}
As seen by Theorem \ref{thm:5.1}, $\QL (C)$ is the union of all $E \in \Max (C)$, and thus \eqref{eq:5.3} implies that
\begin{equation}\label{eq:5.4} \QL (C) = \QL (\tlC ) \end{equation}
 for any (nonempty) quasilinear subset $C$ of $\Ray (V)$.

\begin{thm}\label{thm:5.3} For every quasilinear subset $C$ of $\Ray (V)$
\[ \QL (\QL (C)) = \QL (\QL (\tlC )) = \tlC . \]
\end{thm}

\begin{proof}
 We conclude from Theorem \ref{thm:5.1} and Corollary \ref{cor:5.2} that
\[ \QL (\QL (C)) = \QL \bigg(\bigcup\limits_{i \in I} C_i\bigg)  = \bigcap\limits_{i \in I} \QL (C_i) = \bigcap\limits_{i \in I} C_i = \tlC . \]
Furthermore $\QL (\tlC ) = \QL (C)$ by \eqref{eq:5.4} and so $\QL (\QL (\tlC )) = \tlC $. \end{proof}

The set $\QL (\QL (C))$ had been named  the $\QL$-saturum of $C$ in \S \ref{sec:3}. Consequently, we  say that $C$ is QL-\textbf{saturated} if $C = \QL (\QL (C))$, i.e., $C = \tlC $.
That is, in notation \eqref{eq:3.10},
\begin{equation}\label{eq:5.5} \tlC  = \sat_{\QL} (C) \end{equation}
for every quasilinear subset $C$ of $\Ray (V)$.
Clearly, the QL-saturum of ~$C$ is convex.

\begin{thm}\label{thm:5.4} Given nonempty quasilinear subsets $C$ and $D$ of $\Ray (V)$, the following assertions are equivalent.
\begin{enumerate} \eroman
  \item  $\tlC  \subset \tlD$,
\item $\Max (D) \subset \Max (C)$,
\item $\QL (D) \subset \QL (C)$.

\end{enumerate}
\end{thm}

\begin{proof}
 (ii) $\Rightarrow$ (i): Evident,  since $\tlC $ and $\tlD$ are the intersections of the families $\Max (C)$ and $\Max (D)$,  respectively.\pSkip
 (i) $\Rightarrow$ (ii): $\tlC  \subset \tlD$ implies that $\Max (\tlC ) \supset \Max (\tlD)$ directly by the definition of $\Max (\tlC )$ and $\Max (\tlD)$ (cf.~\eqref{eq:5.2}). This means $\Max (C) \supset \Max (D)$ by \eqref{eq:5.3}.\pSkip
(ii) $\Rightarrow$ (iii): Immediate,  since $\QL (C)$ and $\QL (D)$ are the unions of the families of sets $\Max (C)$ and $\Max (D)$ (Theorem~\ref{thm:5.1}).\pSkip
(iii) $\Rightarrow$ (ii): Given a maximal quasilinear set $E$ in $\Ray (V)$ we have the following chain of implications:
$E \in \Max (D) \Leftrightarrow D \subset E \Rightarrow E = \QL (E) \subset \QL (D) \xLongrightarrow[\text{(iii)}]{\text{}} E \subset \QL (C) \Rightarrow \tlC  = \QL (\QL (C)) \subset \QL (E) = E \Rightarrow E \in \Max (\tlC ) = \Max (C)$. This proves that $\Max (D) \subset \Max (C)$. \end{proof}

We state a consequence of this theorem for the maximal QL-enlargement $E(C)$ of a convex quasilinear set   $C$ in $\Ray(V).$
\begin{cor}\label{cor:5.5}
$E(C) \subset \tlC$, and
$E(\tlC) =  \tlC$, i.e., $\tlC$ has no proper QL-enlargements.
\end{cor}
\begin{proof} $E(C)$ is the convex hull of the union of the sets $\sat_{\QL} (X) = \widetilde{\{ X \} } $ with $X$ running through $C$ (cf. Remark \ref{rem:4.5}). Since by Theorem \ref{thm:5.4} $\widetilde{\{ X \} } \subset \tlC$ for each $X \in C$, we conclude that $E(C) \subset \tlC$. It follows that
$E(\tlC) = {(\tlC)}^{\sim} = \tlC$.
\end{proof}

\section{Convexity of QL-stars}\label{sec:13}

We return to the assumptions made in \S \ref{sec:2}. To repeat, $V$ is a module over a supertropical semiring $R$, whose ghost ideal $eR$ is a nontrivial semifield, and $(q, b)$ is a quadratic pair on ~$V$ with $q$ anisotropic. Thus, for any two rays $X, Y$ in $V$ we have a well defined CS-ratio $\CS (X, Y)$, such that
\[ \CS (X, Y) = \CS (x, y) = \frac{e b (x, y)^2}{e q (x) q (y)} \]
for $x \in X$, $y \in Y$. In addition, we assume that $e \mathcal{T} = \mathcal{G}$, but we do not require  that every element of $\mathcal{T}$ is a unit in $R$, which would mean that $R$ is a tangible supersemifield.

We ask, in which cases a given $\QL$-star $\QL (X)$ is convex in $\Ray (V)$. There is no serious problem if $eR$ is a dense semifield.

\begin{thm}\label{thm:13.1}
 If $eR$ is a dense semifield, then $\QL (X)$ is convex for \textbf{every} $X \in \Ray (V)$.\end{thm}
\begin{proof} It follows from \cite[Theorem 6.7]{QF2} that \textit{now $\QL (X)$ is the set of all $Y \in \Ray (V)$ with $\CS (X, Y) \leq e$}. If $Y_1, Y_2$ are rays in $V$ with $\CS (X, Y_1) \leq e$, $\CS (X, Y_2) \leq e$, then for every $Z \in [X, Y]$ also $\CS (X, Z) \leq e$ due to \cite[Theorem~7.7.a]{QF2} (a special case of the subadditivity theorem \cite[Theorem~3.6.a]{QF2}). Thus $[Y_1, Y_2] \subset \QL (X)$. \end{proof}

We turn to the case that the nontrivial semifield $eR$ is discrete. Then the totally ordered set $eR$ contains a smallest element $c_0 > eR$, and, as known from \cite[Theorem~6.7]{QF2}, a pair of rays $(X, Y)$ is quasilinear, if either $\CS (X, Y) \leq e$, or $\CS (X, Y) = c_0$ and both rays $X, Y$ are $g$-isotropic, i.e., $q (X)$ and $q (Y)$ are subsets of $\mathcal{G}$. In the latter case the pair $(X, Y)$ is called \textbf{exotic quasilinear} \cite[Definition~6.6]{QF2}.

\begin{thm}\label{thm:13.2}
 The QL-stars of all $g$-anisotropic rays in $V$ are convex.
\end{thm}
\begin{proof} If $X$ is an $g$-anisotropic ray in $V$, then there does not exist $Y \in \Ray (V)$ such that the pair $(X, Y)$ is exotic quaislinear. Thus
\[ \QL (X) = \{ Y \in \Ray (V) \vert \CS (X, Y) \leq e \}. \]
We conclude as in the proof of Theorem \ref{thm:13.1} that $\QL (X)$ is convex. \end{proof}

If $X$ is $g$-isotropic, it may happen that $\QL (X)$ is not convex.

\begin{examp}
\label{exmp:13.3} Assume that $\veps_1, \veps_2, \veps_3$ are rays in $V$ with $q (\veps_i) = e$, $b (\veps_i, \veps_j) = \gamma \in \mathcal{T}$ with $e \gamma = c_0$. $(1 \leq i < j \leq 3)$. (Note that this situation can be easily realised for $V$ a free module with base $\veps_1, \veps_2, \veps_3$.) Let $X_i : = \ray (\veps_i)$, $Y_i : = \ray (\veps_j + \veps_k)$, where $i, j, k$ is a permutation of $\{ 1, 2, 3 \}$. Then $Y_i \in [X_j, X_k]$. We compute
$$\begin{array}{lll}
   q (\veps_j + \veps_k)  = e + e + \gamma = \gamma, & &
    b (\veps_i, \veps_j + \veps_k) = \gamma + \gamma = c_0, \\[1mm]
    \CS (X_i, Y_i)  = \frac{e \gamma^2}{e \gamma} = c_0, \quad & &
    \CS (X_j, X_k)  = \frac{c_0^2}{e} = c_0.
\end{array}$$
Thus the pairs $(X_i, X_j)$ and $(X_i, X_k)$ are exotic quasilinear, and so $X_j, X_k \subset \QL (X_i)$, while~ $Y_i$ is $g$-anisotropic (i.e., not $g$-isotropic). Thus $Y_i \in \QL (X_i)$, but $Y_i \notin [X_j, X_k]$. We have the following picture with three non-convex QL-stars $\QL (X_1)$, $\QL (X_2)$, and $\QL (X_3)$.

\[\scalebox{0.7}{
\begin{pspicture}(-2,-1)(8,5)
% Arbeitseinstellungen
%\psgrid[subgridwidth=0.5pt,subgriddiv=10] \psset{showpoints=true}
\psline[arrowsize=0.2]{*-*}(7,0)(3,4.5)
\psline[arrowsize=0.2]{*-*}(-1,0)(3,4.5)
\psline[arrowsize=0.2]{*-*}(7,0)(-1,0)
\psline[arrowsize=0.2]{*-*}(7,0)(1.0,2.25)
\psline[arrowsize=0.2]{*-*}(-1,0)(5.0,2.3)
\psline[arrowsize=0.2]{*-*}(3,4.5)(3.0,0.0)
\rput[l](2.8,5){$X_3$}
\rput[l](0.3,2.5){$Y_2$}
\rput[l](5.3,2.5){$Y_1$}
\rput[l](-1.5,-0.4){$X_1$}
\rput[l](2.8,-0.4){$Y_3$}
\rput[l](7.1,-0.4){$X_2$}
\rput[l](2.5,1.0){$Z$}
\end{pspicture}}
\]
The three ray intervals $[X_1, Y_1]$, $[X_2, Y_2]$, and $[X_3, Y_3]$ meet at $Z = \ray (\veps_1 + \veps_2 + \veps_3)$.
\end{examp}
\begin{thm}\label{thm:13.4}
 Assume that $X_1, X_2, X_3$ are rays in $V$ where $X_1$ is  $g$-isotropic, $X_3$ is $g$-anisotropic, and
\[ \CS (X_1, X_2) \leq \CS (X_1, X_3) = c_0. \]
Then $\QL (X_1)$ is not convex, namely there exist $g$-isotropic rays $Y_1, Y_2$ and a $g$-anisotropic ray $Z$ with $Z \in [Y_1, Y_2] \subset [X_1, X_2[$ and $\CS (X_1, Y_1) = \CS (X_1, Y_2) = \CS (X_1, Z) = c_0$, whence $Y_1, Y_2 \in \QL (X_1)$, $Z \not\in \QL (X_1)$.
\end{thm}
\begin{proof} We choose vectors $\veps_i \in V$ for which $X_i = \ray (\veps_i)$,
$$\begin{array}{ll}
 \al_1 : = q (\veps_1) \in \tG, & \al_2 : = q (\veps_2) \in R \setminus \{ 0 \},\\[1mm]
 \al_3 : = q (\veps_3) \in \tT, & \al_{23} : = b (\veps_2, \veps_3) \in R.
\end{array}
$$
The rays $X$ in $[X_2, X_3[$ have a presentation $X = \ray (\veps_2 + \lm \veps_3)$ with $\lm$ running through $R$. If~ $e \lm$ is big enough, then

$$ \begin{array}{ll}
  q (\veps_2 + \lm \veps_3) & = \al_2 + \lm \al_{23} + \lm^2 \al_3 = \lm^2 \al_3,\\[1mm]
\CS (\veps_1, \veps_2 + \lm \veps_3) & = \frac{\al_{12}^2 + \lm^2 \al_{13}^2}{\al_1 \lm^2 \al_3} = \frac{\al_{12}^2}{\lm^2 \al_1 \al_2} \frac{\al_2}{\al_3} + \frac{\al_{13}^2}{\al_1 \al_3}\\ &  = \CS (\veps_1, \veps_3) = c_0. \end{array}$$
More precisely, this holds if $\lm \geq_{\nu} \lm_0$ for some $\lm_0$ with
\[ \lm_0^2 \geq_{\nu} \max \left( \frac{\al_2}{\al_3}, \frac{\al_3^2}{\al_{23}^2} \right). \]
Now choose scalars $\lm_1 >_{\nu} \rho >_{\nu} \lm_2 \geq_{\nu} \lm_0$ where  $\lm_1, \lm_2 \in \tG$, $\rho \in \tT$, and take $Y_i = \ray (\veps_2 + \lm_i \veps_3)$, $i = 1, 2$, $Z = \ray (\veps_2 + \rho \veps_3)$. \end{proof}

\section{Enlargements of QL-paths, and bridges to find short QL-paths}\label{sec:14}

In this section we only assume, that the pair $(R, V)$ is ray-admissible (cf. \eqref{eq:3.1}, \eqref{eq:3.2}), and that $(q, b)$ is a quadratic pair on $V$ with $q$ anisotropic.

\begin{defn}\label{def:14.1} A \textbf{QL-path} in $\Ray(V)$ is a sequence $(X_0, X_1, \dots, X_n)$ of rays in $V$ such that every pair $(X_i, X_{i+1})$, $0 \leq i < n$, is quasilinear (equivalently, that the closed interval $[X_i, X_{i+1}]$ is quasilinear). We say that the QL-path $(X_0, \dots, X_n)$ has length $n$ and runs from $X_0$ to $X_n$. \end{defn}

This definition has a graph theoretic flavor.

\begin{defn}\label{def:14.2} We define a (simple, undirected) graph $\Gamma_{\QL} (V, q)$ to be the graph whose vertices are the rays in $V$, and its edges are the quasilinear pairs $(X, Y)$ of rays. For formal reasons we admit loops in $\Gamma_{\QL} (V, q)$. For every $X \in \Ray (V)$ we have a loop $(X, X)$ due to the fact that $\CS (X, X) \leq e$, cf. \cite[Eq. (1.9)]{QF1}.\footnote{
%More precisely, $X_i \sim_{\QL} Y_i$. We can always replace a vertex $X$ by a QL-equivalent ray.
If $(q, b)$ is balanced, then $\CS (X, X) = e$ for every $X$ \cite[Eq. (1.10)]{QF1}.} We call $\Gamma_{\QL} (V, q)$ the \textbf{quasilinear graph} of ~$(V, q)$.
\end{defn}
\noi  Note that this graph does not depend on the choice of the companion $b$ of $q$, since the sets $\QL (X)$ are independent of the choice of $b$.

%A fundamental task in the theory of supertropical quadratic forms is to describe the (path) components of $\Gamma_{\QL} (V, q)$ and to extract  information on their diameters. Due to limitation of space and time we leave these topics for future research. But we provide some preparations for this.
We define ``\textit{enlargements}'' of a given QL-path $(X_0, \dots, X_n)$ and use  them to develop procedures for replacing  $(X_0, \dots, X_n)$ by a path of shorter length from $X_0$ to $X_n$ under suitable conditions.

\begin{notation}\label{notat:14.3}

 We refine the graph $\Gamma_{\QL} (V, q)$ by replacing an edge $X \,\raisebox{0.15cm}{\hbox to 0.5cm{\hrulefill}}\, Y$ by an arrow $X \to Y$ in the case that $X \preceq_{\QL} Y$, i.e., $\QL (X) \subset \QL (Y)$, and consequently replace $X \,\raisebox{0.15cm}{\hbox to 0.5cm{\hrulefill}}\, Y$ by an arrow with two heads $X \leftrightarrow Y$, if $X \sim_{\QL} Y$, i.e., $\QL (X) = \QL (Y)$.
 %$\{$Two arrows $\leftrightarrows$ would be more correct.$\}$
 But most often we then abusively identify  $X= Y$, since in all matters below a vertex $X$ can be replaced by a QL-equivalent ray.
 We call the new diagram the \textbf{decorated quasilinear graph} $\tlGm_{\QL} (V, q)$ of $(V, q)$. \end{notation}

\begin{defn}\label{def:14.4} Given two QL-paths $(X_0,X_1, \dots, X_n)$ and $(Y_0, Y_1,\dots, Y_n)$ of same length~ $n$, we say that $(Y_0, Y_1, \dots, Y_n)$ is an \textbf{enlargement} of $(X_0, X_1,\dots, X_n)$, if $\QL (Y_i) \supset \QL (X_i)$ for $0 \leq i \leq n$. Then we have the following subdiagram of $\tlGm_{\QL} (V, q)$:
\begin{equation}\label{eq:14.1}
\setlength{\arraycolsep}{0.2cm}
\begin{array}{ccccccccc}
\Rnode{a}{Y_0} & \raisebox{1mm}[4mm][2mm]{\Rnode{b}{\hbox to 0.6cm{\hrulefill}}} & \Rnode{c}{Y_1} & \raisebox{1mm}[4mm][2mm]{\Rnode{d}{\hbox to 0.6cm{\hrulefill}}} & \Rnode{e}{Y_2} &
\Rnode{g}{\raisebox{1mm}[4mm][2mm]{\hbox to 0.6cm{\hrulefill} \dots \hbox to 0.6cm{\hrulefill}}} & \Rnode{i}{Y_n}
\\[1.1cm]
\Rnode{j}{X_0} & \raisebox{1mm}[4mm][2mm]{\Rnode{k}{\hbox to 0.6cm{\hrulefill}}} & \Rnode{l}{X_1} & \raisebox{1mm}[4mm][2mm]{\Rnode{m}{\hbox to 0.6cm{\hrulefill}}} & \Rnode{n}{X_2} &
\Rnode{g}{\raisebox{1mm}[4mm][2mm]{\hbox to 0.6cm{\hrulefill} \dots \hbox to 0.6cm{\hrulefill}}} & \Rnode{r}{X_n}
%\\[1.1cm]
\end{array}
\psset{nodesep=5pt,arrows=->} \everypsbox{\scriptstyle}
\ncLine{a}{j} \ncLine{c}{l} \ncLine{e}{n}
\ncLine{i}{r}
\end{equation}
\end{defn}
Thus,  by using the mother set $\{ Y_i, Y_{i+1} \}$, we have enlarged the quasilinear interval $[X_i, X_{i+1}]$ to the convex hull of $\{ X_i, X_{i+1}, Y_i, Y_{i+1} \}$. Note that  $X_i = Y_i$, if the associated disjoint mother set is $\{ Y_{i+1} \}$, and that $X_{i+1} = Y_{i+1}$ if this set is $\{ Y_i \}$.

In the diagram \eqref{eq:14.1} we can always enrich a square
$$\setlength{\arraycolsep}{0.1cm}
\begin{array}{ccc}
\Rnode{a}{Y_i} & \raisebox{1mm}[4mm][2mm]{\Rnode{b}{\hbox to 0.6cm{\hrulefill}}} & \Rnode{c}{Y_{i+1}} \\[1cm]
\Rnode{d}{X_i} & \raisebox{1mm}[4mm][2mm]{\Rnode{e}{\hbox to 0.6cm{\hrulefill}}} & \Rnode{f}{X_{i+1}} \\[0.11cm]
\end{array}
\psset{nodesep=5pt,arrows=->} \everypsbox{\scriptstyle}
\ncLine{a}{d} \ncLine{c}{f}
$$
to the subdiagram
\begin{equation}\label{eq:14.2}
\setlength{\arraycolsep}{0.1cm}
\begin{array}{ccc}
\Rnode{a}{Y_i} & \raisebox{1mm}[4mm][2mm]{\Rnode{b}{\hbox to 0.6cm{\hrulefill}}} & \Rnode{c}{Y_{i+1}} \\[1cm]
\Rnode{d}{X_i} & \raisebox{1mm}[4mm][2mm]{\Rnode{e}{\hbox to 0.6cm{\hrulefill}}} & \Rnode{f}{X_{i+1}} \\[0.1cm]
\end{array}
\psset{nodesep=5pt,arrows=->} \everypsbox{\scriptstyle}
\ncLine{a}{d} \ncLine{c}{f}
\psset{nodesep=5pt,arrows=-} \everypsbox{\scriptstyle}
\ncline{a}{f} \ncline{c}{d}
\end{equation}
of $\tlGm_{\QL} (V, q)$.

In other words, $X_i \in \QL (Y_{i+1})$ and $X_{i+1} \in \QL (Y_i)$. Moreover, it may happen, say, if $0 < i < n$, that there are indices $j < i-1$ and $k > i+1$ with $X_j \in \QL (Y_i)$, $X_k \in \QL (Y_i)$. Then we have the following subdiagram of $\tlGm_{\QL} (V, q)$ with $Y : = Y_i$

\begin{equation}\label{eq:14.3}
\setlength{\arraycolsep}{0.1cm}
\begin{array}{ccccccc}
               &      &      & \Rnode{a}{Y} &        &        &       \\[1cm]
\Rnode{b}{X_0} & \Rnode{c}{\raisebox{1mm}[4mm][2mm]{\hbox to 0.6cm{\hrulefill} \dots \hbox to 0.6cm{\hrulefill}}} & \Rnode{d}{X_j} & \qquad \Rnode{e}{X_i} \qquad & \Rnode{f}{X_k} &
\Rnode{g}{\raisebox{1mm}[4mm][2mm]{\hbox to 0.6cm{\hrulefill} \dots \hbox to 0.6cm{\hrulefill}}} & \Rnode{h}{X_n},  \\[0.1cm]
\end{array}
\psset{nodesep=5pt,arrows=->} \everypsbox{\scriptstyle}
\ncLine{a}{e}
\psset{nodesep=5pt,arrows=-} \everypsbox{\scriptstyle}
\ncline{a}{d} \ncline{a}{f}
\end{equation}

\noi
which gives  a QL-path $(X_0, \dots, X_j, Y, X_k, \dots X_n)$ from $X_0$ to $X_n$ of shorter length $j + 1 + n-k = n - (k-j-1)$. Also there may exist two different rays $Y', Y''$ with $\QL (Y') \supset \QL (X_i)$, $\QL (Y'') \supset \QL (X_i)$, so that we have an index $r < j$ and an index $s > k$ with $X_r \in \QL (Y')$, $X_s \in \QL (Y'')$. Then we have a subdiagram

\begin{equation}\label{eq:14.4}
\setlength{\arraycolsep}{0.1cm}
\begin{array}{ccccccccc}
        &    &    & \Rnode{a}{Y'} & \raisebox{1mm}[4mm][2mm]{\Rnode{b}{\hbox to 0.6cm{\hrulefill}}} & \Rnode{c}{Y''}  &   &   &   \\[1cm]
\Rnode{d}{X_0} & \Rnode{e}{\raisebox{1mm}[4mm][2mm]{\hbox to 0.6cm{\hrulefill} \dots \hbox to 0.6cm{\hrulefill}}} & \Rnode{f}{X_r} & \Rnode{g}{} & \Rnode{h}{X_i} & \Rnode{i}{} & \Rnode{j}{X_s} &
\Rnode{k}{\raisebox{1mm}[4mm][2mm]{\hbox to 0.6cm{\hrulefill} \dots \hbox to 0.6cm{\hrulefill}}} & \Rnode{l}{X_n}, \\[0.1cm]
\end{array}
\psset{nodesep=5pt,arrows=->} \everypsbox{\scriptstyle}
\ncLine{a}{h} \ncLine{c}{h}
\psset{nodesep=5pt,arrows=-} \everypsbox{\scriptstyle}
\ncline{a}{f} \ncline{c}{j}
\end{equation}

\noi
of $\tlGm_{\QL} (q)$ which gives a path from $X_0$ to $X_n$ of even smaller length $r + 2 + (n-s) = n - (s-r-2). < n- (k-j-1)$. In more imaginative terms, we have built ''bridges'' in \eqref{eq:14.3} and \eqref{eq:14.4} to span the subpaths $(X_j, X_{j+1}, \dots, X_k)$ and $(X_r, \dots, X_s)$ respectively by use of one or two rays in $\sat_{\QL} (X_i)$ as ``pillars''.

We are ready for a formal definition of a bridge. Assume that $(X_0, X_1, \dots, X_n)$ is any QL-path in $\Ray (V)$.

\begin{defn}\label{def:14.5} A \textbf{bridge over} $(X_0, X_1, \dots, X_n)$ (or \textbf{spanning} $(X_0, X_1, \dots, X_n)$) is a QL-path $(X_0, Y_1, \dots, Y_m, X_n)$ together with a sequence $0 \leq c (1) < c (2) < \dots < c (m) \leq n$ such that
\[ Y_r \in \sat_{\QL} (X_{c(r)}) \]
for $1 \leq r \leq m$, and furthermore $c(2) \geq 2$ in the case $c(1) = 0$, and $c (m-1) \leq n-2$, in the case $c(m) = n$. \end{defn}
Note that we do not exclude the possibility that $Y_r = X_{c(r)}$ for some indices $r$.

\begin{comm}\label{comm:14.6} In the case $0 < c(1), c(m) < n$, a bridge over $(X_0,X_1, \dots, X_n)$ is given by a diagram
\begin{equation}\label{eq:14.5}
\setlength{\arraycolsep}{0.05cm}
\begin{array}{ccccccccc}
       &    &  \Rnode{a}{Y_1} &  \Rnode{b}{\raisebox{1mm}[4mm][2mm]{\hbox to 1.8cm{\hrulefill}}}
       %\Rnode{b}{\hbox to 1.8cm{\hrulefill}}
        & \Rnode{c}{Y_2} & \Rnode{d}{\raisebox{1mm}[4mm][2mm]{\hbox to 0.6cm{\hrulefill} \dots \hbox to 0.6cm{\hrulefill}}}  & \Rnode{e}{Y_m}  &   &   \\[1cm]
\Rnode{f}{X_0} & %\Rnode{b}{\raisebox{1mm}[4mm][2mm]{\hbox to 1.8cm{\hrulefill}}}
\Rnode{g}{\raisebox{1mm}[4mm][2mm]{\hbox to 0.6cm{\hrulefill} \dots \hbox to 0.6cm{\hrulefill}}}
& \Rnode{h}{X_{c(1)}} & % \Rnode{b}{\hbox to 1.8cm{\hrulefill}}
\Rnode{i}{\raisebox{1mm}[4mm][2mm]{\hbox to 0.6cm{\hrulefill} \dots \hbox to 0.6cm{\hrulefill}}}
& \Rnode{j}{X_{c(2)}} & \Rnode{k}{\raisebox{1mm}[4mm][2mm]{\hbox to 0.6cm{\hrulefill} \dots \hbox to 0.6cm{\hrulefill}}} & \Rnode{l}{X_{c(m)}} &
\Rnode{m}{\raisebox{1mm}[4mm][2mm]{\hbox to 0.6cm{\hrulefill} \dots \hbox to 0.6cm{\hrulefill}}} & \Rnode{n}{X_n}, \\[0.1cm]
\end{array}
\psset{nodesep=5pt,arrows=->} \everypsbox{\scriptstyle}
\ncLine{a}{h} \ncLine{c}{j} \ncline{e}{l}
\psset{nodesep=5pt,arrows=-} \everypsbox{\scriptstyle}
\ncline{a}{f} \ncline{e}{n}
%\psset{nodesep=5pt,angle=-50}
%\ncarc[arcangle=-40]{->}{a}{f}
\end{equation}
in $\tlGm_{\QL} (V, q)$, while, if say $c(1) = 0$, $c(m) < n$, we have a diagram

\begin{equation}\label{eq:14.6}
\setlength{\arraycolsep}{0.05cm}
\begin{array}{ccccccccc}
\Rnode{a}{Y_1} & \raisebox{1mm}[4mm][2mm]{\Rnode{b}{\hbox to 1.8cm{\hrulefill}}} & \Rnode{c}{Y_2} & \raisebox{1mm}[4mm][2mm]{\Rnode{d}{\hbox to 1.8cm{\hrulefill}}} & \Rnode{e}{Y_3} & \Rnode{f}{\raisebox{1mm}[4mm][2mm]{\hbox to 0.6cm{\hrulefill} \dots \hbox to 0.6cm{\hrulefill}}} & \Rnode{g}{Y_m}  &   &   \\[1cm]
\Rnode{h}{X_0} & \Rnode{i}{\raisebox{1mm}[4mm][2mm]{\hbox to 0.6cm{\hrulefill} \dots \hbox to 0.6cm{\hrulefill}}} & \Rnode{j}{X_{c(2)}} & \Rnode{k}{\raisebox{1mm}[4mm][2mm]{\hbox to 0.6cm{\hrulefill} \dots \hbox to 0.6cm{\hrulefill}}} & \Rnode{l}{X_{c(3)}} & \Rnode{m}{\raisebox{1mm}[4mm][2mm]{\hbox to 0.6cm{\hrulefill} \dots \hbox to 0.6cm{\hrulefill}}} & \Rnode{n}{X_{c(m)}} &
\Rnode{o}{\raisebox{1mm}[4mm][2mm]{\hbox to 0.6cm{\hrulefill} \dots \hbox to 0.6cm{\hrulefill}}} & \Rnode{p}{X_n}, \\[0.1cm]
\end{array}
\psset{nodesep=5pt,arrows=->} \everypsbox{\scriptstyle}
\ncLine{c}{j} \ncLine{e}{l} \ncline{g}{n}
\psset{nodesep=5pt,arrows=-} \everypsbox{\scriptstyle}
\ncline{a}{h} \ncline{g}{p}
%\psset{nodesep=5pt,angle=-50}
\ncarc[arcangle=-30]{->}{a}{h}
\end{equation}
If we would allow here $c(2) = 1$, we could omit the ray $Y_1$ in the QL-path $(X_0, Y_1, \dots, Y_m, X_n)$ (cf. \eqref{eq:14.2}) and would obtain for free the shorter bridge
\begin{equation}\label{eq:14.7}
\setlength{\arraycolsep}{0.05cm}
\begin{array}{ccccccccc}
       &    &  \Rnode{a}{Y_2} & \Rnode{b}{\raisebox{1mm}[4mm][2mm]{\hbox to 1.8cm{\hrulefill}}} & \Rnode{c}{Y_3} & \Rnode{d}{\raisebox{1mm}[4mm][2mm]{\hbox to 0.6cm{\hrulefill} \dots \hbox to 0.6cm{\hrulefill}}}  & \Rnode{e}{Y_m}  &   &   \\[1cm]
\Rnode{f}{X_0} & \Rnode{b}{\raisebox{1mm}[4mm][2mm]{\hbox to 1.8cm{\hrulefill}}}
%\Rnode{g}{\raisebox{1mm}[4mm][2mm]{\hbox to 0.6cm{\hrulefill} \dots \hbox to 0.6cm{\hrulefill}}}
& \Rnode{h}{X_1} & \Rnode{i}{\raisebox{1mm}[4mm][2mm]{\hbox to 0.6cm{\hrulefill} \dots \hbox to 0.6cm{\hrulefill}}} & \Rnode{j}{X_{c(3)}} & \Rnode{k}{\raisebox{1mm}[4mm][2mm]{\hbox to 0.6cm{\hrulefill} \dots \hbox to 0.6cm{\hrulefill}}} & \Rnode{l}{X_{c(m)}} &
\Rnode{m}{\raisebox{1mm}[4mm][2mm]{\hbox to 0.6cm{\hrulefill} \dots \hbox to 0.6cm{\hrulefill}}} & \Rnode{n}{X_n}. \\[0.1cm]
\end{array}
\psset{nodesep=5pt,arrows=->} \everypsbox{\scriptstyle}
\ncLine{a}{h} \ncLine{c}{j} \ncline{e}{l}
\psset{nodesep=5pt,arrows=-} \everypsbox{\scriptstyle}
\ncline{a}{f} \ncline{e}{n}
%\psset{nodesep=5pt,angle=-50}
%\ncarc[arcangle=-40]{->}{a}{f}
\end{equation}
We want to discard this annoying triviality.
\end{comm}

\begin{comm}\label{comm:14.7}  Our formal definition of bridges does not include the ``bridge'' \eqref{eq:14.4} with a doubled pillar. But it includes an equivalent object. Assume that $X_{i-1} = X_i$ in a given path $(X_0, X_1, \dots, X_n)$, i.e., the QL-path $(X_0, \dots, X_n)$ which contains the loop $(X_i, X_i)$ of $\Gamma_{\QL} (V, q)$. Then the interval $[X_{i-1}, X_i]$ shrinks to the one-point set $\{ X_i \}$. The diagram \eqref{eq:14.4} shows in essence the same objects as the bridge
\begin{equation}\label{eq:14.8}
\setlength{\arraycolsep}{0.1cm}
\begin{array}{ccccccccc}
        &    &    & \Rnode{a}{Y'} & \raisebox{1mm}[4mm][2mm]{\Rnode{b}{\hbox to 0.8cm{\hrulefill}}} & \Rnode{c}{Y}  &   &   &   \\[1cm]
\Rnode{d}{X_0 \; \raisebox{1mm}[4mm][2mm]{\hbox to 0.6cm{\hrulefill} \dots \hbox to 0.6cm{\hrulefill}}} & \Rnode{e}{X_r} & \Rnode{f}{\qquad} & \Rnode{g}{X_{i-1}} & \raisebox{1mm}[4mm][2mm]{\Rnode{h}{\hbox to 0.6cm{\hrulefill}}} & \Rnode{i}{X_i} & \Rnode{j}{\qquad} &
\Rnode{k}{X_s} & \Rnode{l}{\raisebox{1mm}[4mm][2mm]{\hbox to 0.6cm{\hrulefill} \dots \hbox to 0.6cm{\hrulefill}} \; X_n}.  \\[0.1cm]
\end{array}
\psset{nodesep=5pt,arrows=->} \everypsbox{\scriptstyle}
\ncLine{a}{g} \ncLine{c}{i}
\psset{nodesep=5pt,arrows=-} \everypsbox{\scriptstyle}
\ncline{a}{e} \ncline{c}{k}
\end{equation}\end{comm}

We note an important fact, immediately obtained  from the definition 14.5 of bridges.

\begin{prop}\label{prop:14.8}
 Assume that $(X_0, Y_1, \dots, Y_m, X_n)$ is a bridge over $(X_0, X_1, \dots, X_n)$. Then any bridge $(X_0, Z_1, \dots, Z_q, X_n)$ spanning the QL-path $(X_0, Y_1, \dots, Y_m, X_n)$ is again a bridge over $(X_0, X_1, \dots, X_n)$. \end{prop}

We describe a procedure to shorten a given QL-path $(X_0, X_1, \dots, X_n)$ without yet using enlargements of $(X_0,X_1, \dots, X_n)$.\footnote{An analogous procedure can be performed in any simple graph.}

\begin{defn}\label{def:14.9}  A \textbf{basic reduction} of the QL-path $(X_0, \dots, X_n)$ arises as follows. Pick some $i \in [0, n]$.
\begin{enumerate}\ealph
  \item  If there exist indices $k > i + 1$ such that the pair $(X_i, X_k)$ is quasilinear, let $s$ denote the maximal one of these and replace $(X_0, \dots, X_n)$ by $(X_0, \dots, X_i, X_s, \dots, X_n)$ omitting all rays $X_p$ with $i < p < s$.
  \item If there exist indices $j < i-1$ such that $(X_j, X_i)$ is quasilinear, let $r$ denote the minimal of these, and replace $(X_0, \dots, X_n)$ by $(X_0, \dots, X_r, X_i, \dots, X_n)$, omitting all rays $X_p$ with $r < p < i$.

\end{enumerate}
More precisely we call a QL-path $(X_0, \dots, X_i, X_s, \dots, X_n)$ as in (a) an \textbf{f-basic reduction (= forward basic reduction)} of $(X_0, \dots, X_n)$ and a QL-path $(X_0, \dots, X_r, X_i, \dots, X_n)$ as in (b) a \textbf{b-basic reduction (= backward basic reduction)} of $(X_0, \dots, X_n)$. \end{defn}

\begin{defn}\label{def:14.10}  We say that a QL-path $(X_0, X_1, \dots, X_n)$ is \textbf{direct}, if $X_0 \ne X_n$ and there do not exist indices $i, j \in [0, n]$ with $\vert i - j \vert \geq 2$, such that the pair $(X_i, X_j)$ is quasilinear. (Note that this implies $X_i \ne X_{i+1}$ for $0 \leq i < n$.) \end{defn}

The following is obvious from Definitions \ref{def:14.9} and \ref{def:14.10}.

\begin{prop}\label{prop:14.11} $ $
\begin{enumerate} \ealph
  \item
 A QL-path $(X_0, \dots, X_n)$ with $X_0 \ne X_n$ is direct iff no forward basic reduction of $(X_0, \dots, X_n)$ exists, iff no backward basic reduction of $(X_0, \dots, X_n)$ exists.

  \item Any QL-path $(X_0, \dots, X_n)$ with $X_0 \ne X_n$ can be reduced to a direct QL-path by finitely many (at most $n-1$) such reductions.
\end{enumerate}
 \end{prop}

\begin{rem}\label{rem:14.12}
 It may  happen that $(X_0, Y_1, \dots, X_n)$ can be reduced in this way to different direct QL-paths. Assume for example that $n = 6$ and that $(X_0, X_2)$ and  $(X_1, X_5)$ are the only quasilinear pairs $(X_i, X_j)$ with $0 \leq i$, $j \leq n$ and $j-i \geq 2$. Then omitting the ray $X_1$ in $(X_0, \dots, X_6)$ gives us a direct path of length 5, while omitting $X_2, X_3, X_4$ gives us a direct path of length 3.

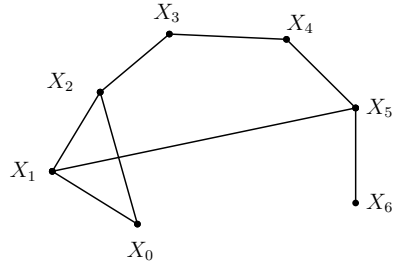
\[\scalebox{0.7}{
\begin{pspicture}(-2,0)(8,5)
% Arbeitseinstellungen
%\psgrid[subgridwidth=0.5pt,subgriddiv=10] \psset{showpoints=true}

\psline[arrowsize=0.2]{*-*}(4.0,4.5)(1.8,4.6)
\psline[arrowsize=0.2]{*-*}(0.5,3.5)(1.8,4.6)
\psline[arrowsize=0.2]{*-*}(0.5,3.5)(-0.4,2.0)
\psline[arrowsize=0.2]{*-*}(-0.4,2.0)(1.2,1.0)
\psline[arrowsize=0.2]{*-*}(4.0,4.5)(5.3,3.2)
\psline[arrowsize=0.2]{*-*}(5.3,3.2)(5.3,1.4)
\psline[arrowsize=0.2]{*-*}(5.3,3.2)(-0.4,2.0)
\psline[arrowsize=0.2]{*-*}(1.2,1.0)(0.5,3.5)
\rput[l](1.5,5){$X_3$}
\rput[l](4.0,4.8){$X_4$}
\rput[l](-0.5,3.7){$X_2$}
\rput[l](5.5,3.2){$X_5$}
\rput[l](-1.2,2.0){$X_1$}
\rput[l](5.5,1.4){$X_6$}
\rput[l](1.0,0.5){$X_0$}
\end{pspicture}
}
\]
\end{rem}

We now describe a procedure to shorten a QL-path by use of enlargements.

\begin{proc}\label{proc:14.13} Given a QL-path $(X_0, \dots, X_n)$ with $X_0 \ne X_n$, we pick an index $i \in [0, n]$ and choose a ray $Y$ in $\sat_{\QL} (X_i)$, i.e., with $\QL (Y) \supset \QL (X_i)$.
\begin{itemize}
\item[a)] If $i > 0$ and $X_k \in \QL (Y)$ for some $k > i + 1$, let $s$ denote the maximal index $\leq n$ with $X_s \in \QL (Y)$ and build
the bridge

\begin{equation}\label{eq:14.9}
\setlength{\arraycolsep}{0.1cm}
\begin{array}{ccccccc}
               &      &      & \Rnode{a}{Y} &        &        &       \\[1cm]
\Rnode{b}{X_0} & \Rnode{c}{\raisebox{1mm}[4mm][2mm]{\hbox to 0.6cm{\hrulefill} \dots \hbox to 0.6cm{\hrulefill}}} & \Rnode{d}{X_{i-1}} & \qquad \Rnode{e}{X_i} \qquad & \Rnode{f}{X_s} &
\Rnode{g}{\raisebox{1mm}[4mm][2mm]{\hbox to 0.6cm{\hrulefill} \dots \hbox to 0.6cm{\hrulefill}}} & \Rnode{h}{X_n} \\[0.1cm]
\end{array}
\psset{nodesep=5pt,arrows=->} \everypsbox{\scriptstyle}
\ncLine{a}{e}
\psset{nodesep=5pt,arrows=-} \everypsbox{\scriptstyle}
\ncline{a}{d} \ncline{a}{f}
\end{equation}
over $(X_0, \dots, X_n)$. This gives us a path $(X_0, \dots, X_{i-1}, Y, X_s, \dots, X_n)$ with $s > i + 1$ of length $n - (s-i-1)$. If $i = 0$, do the same, provided there is an index $k > 2$ with $X_k \in \QL (Y)$. This gives us a bridge

\begin{equation}\label{eq:14.10}
\setlength{\arraycolsep}{0.05cm}
\qquad
\begin{array}{ccccc}
\Rnode{a}{Y} &    &                &  &    \\[1cm]
\Rnode{b}{X_0} &  \qquad  & \Rnode{c}{X_s} & \Rnode{d}{\raisebox{1mm}[4mm][2mm]{\hbox to 0.6cm{\hrulefill} \dots \hbox to 0.6cm{\hrulefill}}} & \Rnode{e}{X_n} \\[0.1cm]
\end{array}
\psset{nodesep=5pt,arrows=-} \everypsbox{\scriptstyle}
\ncline{a}{b} \ncline{a}{c}
%\psset{nodesep=5pt,angle=-50}
\ncarc[arcangle=-30]{->}{a}{b}
\end{equation}
and a path $(X_0, Y, X_s, \dots, X_n)$ with $s > 2$ of length $n - (s-2)$.
\item[b)] If $i < n$ and there exists an index $j < i - 1$ with $X_j \in \QL (Y)$, let $r$ denote the minimal index with $X_r \in \QL (Y)$ and build the bridge

\begin{equation}\label{eq:14.11}
\setlength{\arraycolsep}{0.1cm}
\begin{array}{ccccccc}
               &      &      & \Rnode{a}{Y} &        &        &       \\[1cm]
\Rnode{b}{X_0} & \Rnode{c}{\raisebox{1mm}[4mm][2mm]{\hbox to 0.6cm{\hrulefill} \dots \hbox to 0.6cm{\hrulefill}}} & \Rnode{d}{X_r} & \qquad \Rnode{e}{X_i} \qquad & \Rnode{f}{X_{i+1}} &
\Rnode{g}{\raisebox{1mm}[4mm][2mm]{\hbox to 0.6cm{\hrulefill} \dots \hbox to 0.6cm{\hrulefill}}} & \Rnode{h}{X_n} \\[0.1cm]
\end{array}
\psset{nodesep=5pt,arrows=->} \everypsbox{\scriptstyle}
\ncLine{a}{e}
\psset{nodesep=5pt,arrows=-} \everypsbox{\scriptstyle}
\ncline{a}{d} \ncline{a}{f}
\end{equation}
over $(X_0, \dots, X_n)$. This gives us a path $(X_0, \dots, X_r, Y, X_{i+1}, \dots, X_n)$ with $r < i-1$ of length $n - (i-1-r)$. If $i = n$, do the same, provided there exists an index $j < n-2$ with $X_j \in \QL (Y)$. This gives us the bridge

\begin{equation}\label{eq:14.12}
\setlength{\arraycolsep}{0.05cm}
\begin{array}{ccccc}
               &    &    &    & \Rnode{a}{Y}   \\[1cm]
\Rnode{b}{X_0} & \Rnode{c}{\raisebox{1mm}[4mm][2mm]{\hbox to 0.6cm{\hrulefill} \dots \hbox to 0.6cm{\hrulefill}}} & \Rnode{d}{X_r} & \qquad & \Rnode{e}{X_n} \\[0.1cm]
\end{array}
\psset{nodesep=5pt,arrows=-} \everypsbox{\scriptstyle}
\ncline{a}{e} \ncline{a}{d}
%\psset{nodesep=5pt,angle=-50}
\ncarc[arcangle=30]{->}{a}{e}
\end{equation}
and the path $(X_0, \dots, X_r, Y, X_n)$ with $r < n-2$ of length $r + 2 < n$. \hfill $\square$

\end{itemize}
\end{proc}

\begin{defn}\label{def:14.14}
 We call the paths so obtained \textbf{elementary reductions} of $(X_0, \dots, X_n)$, more precisely, those obtained in a) \textbf{$f$-elementary} (= forward elementary) reductions, and those obtained in b) $b$-\textbf{elementary} (= backward elementary) reductions of $(X_0, \dots, X_n)$. We further call the bridges \eqref{eq:14.9}--\eqref{eq:14.12} \textbf{elementary bridges} over $(X_0, \dots, X_n)$ (or spanning $(X_0, \dots, X_n)$). \end{defn}

\begin{defn}\label{def:14.15}
 We call a QL-path $(X_0, X_1, \dots, X_n)$ \textbf{optimal}, if $n \geq 3 $, $X_0 \ne X_n$, and the path does not admit any elementary reduction.
\end{defn}
Observe that this implies $X_0 \ne X_n$, since in the case $X_0 = X_n$ we would have a bridge \eqref{eq:14.10} with $Y = X_0$ and $s = n$.
%
%Since an elementary reduction of a QL-path has shorter length than the given one, it is plain that any QL-path of length $n \geq 3$ becomes optimal after finitely many elementary reductions (at most $n-1$ reductions).
%
Since an elementary reduction of a QL-path has shorter length than the given one, it is plain that any QL-path $(X_0, \dots, X_n)$ with $n \geq 3$, $X_0 \ne X_n$ becomes either optimal or direct of length $\leq 2$.

\begin{prop}\label{def:14.16}
 Every optimal QL-path is a direct QL-path.
\end{prop}

\begin{proof} Let $(X_0, \dots, X_n)$ be a QL-path of length $n \geq 3$ which is not direct. We verify that $(X_0, \dots, X_n)$ admits an  elementary reduction, and then will be done.
There are indices $i, j \in [0, n]$ with $j \geq i + 2$ and $(X_i, X_j)$ quasilinear. Fixing $i$, let $s$ denote the maximal such index $j$. If $i > 0$, or if $i = 0$, $s \geq 3$, we have an $f$-elementary reduction obtained by a bridge \eqref{eq:14.9} or \eqref{eq:14.10} with $Y = X_i$. There remains the case $i = 0$, $s = 2$. Now $(X_0, X_2)$ is quasilinear, and so we have a $b$-elementary reduction $(X_0, X_2, \dots, X_n)$ by the bridge \eqref{eq:14.11}, there with $i =2$, $r = 0$, $Y = X_2$. \end{proof}

For any ray $X$ in $V$ let $\htQ (X)$ denote the union of all QL-stars containing $\QL (X)$. In other terms,
\begin{equation}\label{eq:14.13}
 \widehat{\QL} (X) : = \bigcup \{ \QL (Y) \ds  \vert Y \in \sat_{\QL} (X) \}. \end{equation}
It is obvious from the definition of elementary reductions and optimal paths (Definitions ~\ref{def:14.14} and  \ref{def:14.15}) that the following holds.

\begin{schol}\label{schol:14.17}
 A QL-path $(X_0, \dots, X_n)$ with $n \geq 3$,  $X_0 \ne X_n$ is optimal iff
\[ \widehat{\QL} (X_i) \cap \{ X_0, \dots, X_n \} = \{ X_{i-1}, X_i, X_{i-1}\}, \]

\noi
for $0 < i < n$, while
\[ \widehat{\QL} (X_0) \cap \{ X_0, \dots, X_n \} \subset  \{ X_0, X_1 \} \]
and
\[ \widehat{\QL} (X_n) \cap \{ X_0, \dots, X_n \} \subset  \{ X_{n-1}, X_n \} \]

\end{schol}

We look for a characterization of optimal paths by properties of their enlargements.

\begin{thm}\label{thm:14.18}
  Assume that $(X_0, \dots, X_n)$ is a QL-path with $X_0 \ne X_n$. The following are equivalent.
\begin{enumerate}\eroman
  \item
 $(X_0, \dots, X_n)$ is optimal.
\item For every $i \in [0, n]$ and $Y \in \sat_{\QL} (X_i)$ is $(X_0, \dots, X_{i-1}, Y, X_{i+1}, \dots, X_n)$ a direct QL-path.

\item For every $i \in [0, n]$ and $Y \in \sat_{\QL} (X_i)$ there exists a direct enlargement $(Z_0, \dots, Z_n)$ (i.e., an enlargement which is a direct QL-path) of $(X_0, \dots, X_n)$ with $\QL (Z_i) \supset \QL (Y)$.
\end{enumerate}

\end{thm}
\begin{proof} (i) $\Leftrightarrow$ (ii): Evident from Scholium \ref{schol:14.17}.\pSkip
(ii) $\Rightarrow$ (iii): Trivial, since $(X_0, \dots, X_{i-1}, Y, X_{i-1}, \dots, X_n)$ is an enlargement of $(X_0, \dots, X_n)$.\pSkip
(iii) $\Rightarrow$ (ii): Let $Y \in \sat_{\QL} (X_i)$. Suppose there exists $j \in [0, n]$ with $\vert j - i \vert > 1$ and $X_j \in \QL (Y)$. Choose a direct enlargement $(Z_0, \dots, Z_n)$ of $(X_0, \dots, X_n)$ with $\QL (Z_i) \supset \QL (Y)$. Then $X_j \in \QL (Z_i)$, and so $Z_i \in \QL (X_j)$. Since $\QL (X_j) \subset \QL (Z_j)$, this implies $Z_i \in \QL (Z_j)$, contradicting our hypothesis that $(Z_0, \dots, Z_n)$ is direct. \end{proof}

Theorem \ref{thm:14.18} can be stated in a more conceptual way by use of a quasiordering $\preceq_{\QL}$ on the set of all QL-paths of fixed length.

\begin{defn}\label{def:14.19} For any two QL-paths $(X_0, \dots, X_n)$, $(Y_0, \dots, Y_n)$, we say that $(Y_0, \dots, Y_n)$ \textbf{dominates} $(X_0, \dots, X_n)$, and write $(X_0, \dots, X_n) \preceq_{\QL} (Y_0, \dots, Y_n)$, if
$X_i \preceq_{\QL} Y_i$ for all $0 \leq i \leq n$, in other terms, $\QL (X_i) \subset \QL (Y_i)$ for every  $0 \leq i \leq n$. \end{defn}

\begin{thm}\label{thm:14.20}
 A QL-path $(X_0, \dots, X_n)$ is optimal iff the set of all direct enlargements of $(X_0, \dots, X_n)$ is cofinal in the set of all enlargements of  $(X_0, \dots, X_n)$, i.e.,  every enlargement $(Y_0, \dots, Y_n)$ of $(X_0, \dots, X_n)$ is dominated by some enlargement $(Z_0, \dots, Z_n)$ of $(X_0, \dots, X_n)$ which is a direct path.
\end{thm}
\begin{proof} We verify the equivalence of this condition with condition (iii) in Theorem~\ref{thm:14.18}. It is plain that the new condition implies condition (iii). On the other hand, if an enlargement $(Y_0, \dots, Y_n)$ of $(X_0, \dots, X_n)$ is given and (iii) holds, we find for every $i \in [0, n]$ a ray $Z_i$ with $\QL (Z_i) \supset \QL (Y_i)$ and $X_j \not\in \QL (Z_i)$ for $\vert i-j \vert > 1$. Now $(Z_0, \dots, Z_n)$ is a direct enlargement of $(X_0, \dots, X_n)$ dominating $(Y_0, \dots, Y_n)$ (cf. Scholium \ref{schol:14.17}).
\end{proof}

We add an observation which enriches the picture around this theorem.

\begin{prop}\label{prop:14.21}
Assume that $(X_0, \dots, X_n)$ is a QL-path which is dominated by a direct QL-path $(Y_0, \dots, Y_n)$. Then $(X_0, \dots, X_n)$ itself is direct.
\end{prop}

\begin{proof} Assume that $i < j$ are indices in $[0, n]$ such that the pair $(X_i, X_j)$ is quasilinear, in other terms, $X_j \in \QL (X_i)$. We have $\QL (X_i) \subset \QL (Y_i)$ and $\QL (X_j) \subset \QL (Y_j)$, and so $X_j \in \QL (Y_i)$, whence $Y_i \in \QL (X_j)$ and then $Y_i \in \QL (Y_j)$. Since $(Y_0, \dots Y_n)$ is direct, it follows that $j = i + 1$, as desired. \end{proof}

%\begin{oproblem}\label{prob:14.21}
%  Let $(X_0, \dots, X_n)$ be an optimal path.
%\begin{itemize}
%\item[a)] When does it happen that the set of enlargements of $(X_0, \dots, X_n)$, which are optimal paths, is cofinal in the set of all enlargements?
%\item[b)] When does it happen that every enlargement of $(X_0, \dots, X_n)$ is a direct path?
%\end{itemize}
%\end{oproblem}
%
%It appears to us that under special conditions on $(R, V, q)$ both  situations are possible, but that a finer analysis of the convex quasilinear subsets and QL-stars in $\Ray (V)$ is needed to find such conditions. We are not aware of any serious obstruction.

It can happen that every enlargement of an optimal path $(X_0, \dots, X_n)$, $n \geq 3$, is again optimal, cf. \S\ref{sec:18} below.

%\bigskip\noi
%It appears to us that under special conditions on $(R, V, q)$ both these situations are possible, but that a finer analysis of the convex quasilinear subsets and QL-stars in $\Ray (V)$ is needed than given here, to find such conditions. We are not aware of any serious obstruction.

\section{Interplay of the quasilinear ordering on the ray space with direct QL-paths}\label{sec:15}

In this section we study an interplay of the (partial) quasiordering $\preceq_{\QL}$ on $\Ray (V)$ with the direct QL-paths (in  particular the optimal QL-paths) in $V$. As in \S \ref{sec:14} we only assume that $(R, V)$ is ray-admissible and $(q, b)$ is a quadratic pair on $V$ with $q$ anisotropic.

Given a ray $X$ in $V$, we denote the upset and downset of $X$ w.r. to $\preceq_{\QL}$ by $X^\uparrow$ and $X^\downarrow$, i.e.,
\begin{equation}\label{eq:15.1}
X^{\uparrow} = \{ Z \in \Ray (V) \ds \vert \QL (X) \subset \QL (Z) \} = \sat_{\QL} (X)
\end{equation}
in previous terminoloy, and
\begin{equation}\label{eq:15.2}
 X^{\downarrow} = \{ Z \in \Ray (V) \ds \vert \QL (Z) \subset \QL (X) \}. \end{equation}
More generally we denote for any set $S$ in $\Ray (V)$ the up- and downsets of $S$ by $S^{\uparrow}$ and $S^{\downarrow}$, i.e.,
\begin{equation}\label{eq:15.3}
 S^{\uparrow} = \bigcup\limits_{X \in S} X^{\uparrow}, \quad S^{\downarrow} = \bigcup\limits_{X \in S} X^{\downarrow}. \end{equation}

\begin{thm}\label{thm:15.1}

 Assume that $(Z_0, \dots, Z_n)$ is an enlargement of an optimal QL-path \\ $(X_0, \dots, X_n)$ in $V$. Then the sets
\[ \{ Z_0, Z_1 \}^{\uparrow}, \  Z_2^{\uparrow}, \ \dots, \ Z_{n-1}^{\uparrow}, \ \{ Z_{n-1}, Z_n \}^{\uparrow} \]
are mutually disjoint, and so form a partition of the set $\{ Z_0, \dots, Z_n \}^{\uparrow}$.
\end{thm}

\begin{proof} Since $Z_i^{\uparrow} \subset X_i^{\uparrow}$, it suffices to verify this for the optimal QL-path $(X_0, \dots, X_n)$ itself. Suppose that there exists a ray $Y$ with $Y \in X_p^{\uparrow} \cap X_q^{\uparrow}$ for different indices $p, q$ in $[0, n]$. If $0 < p < q <n$ we would obtain a diagram

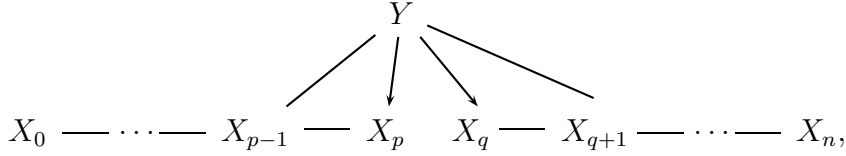
\[
\setlength{\arraycolsep}{0.1cm}
\begin{array}{cccccccc}
               &   &   &      & \Rnode{a}{Y} &        &        &       \\[1cm]
\Rnode{b}{X_0} & \Rnode{c}{\raisebox{1mm}[4mm][2mm]{\hbox to 0.6cm{\hrulefill} \dots \hbox to 0.6cm{\hrulefill}}} & \Rnode{d}{X_{p-1}} & \Rnode{e}{^{\hbox to 0.6cm{\hrulefill}}} & \Rnode{f}{X_p} \quad & \Rnode{g}{X_q \; ^{\hbox to 0.6cm{\hrulefill}}} &
\Rnode{h}{X_{q+1} \; \raisebox{1mm}[4mm][2mm]{\hbox to 0.6cm{\hrulefill} \dots \hbox to 0.6cm{\hrulefill}}} & \Rnode{i}{X_n}, \\[1mm]
\end{array}
\psset{nodesep=5pt,arrows=->} \everypsbox{\scriptstyle}
\ncLine{a}{f} \ncline{a}{g}
\psset{nodesep=5pt,arrows=-
} \everypsbox{\scriptstyle}
\ncline{a}{d} \ncline{a}{h}
\]
and so a QL-path $(X_0, \dots, X_{p-1}, Y, X_{q+1}, \dots, X_n)$ which is both an $f$-elementary and a $b$-elementary reduction of $(X_0, \dots, X_n)$ in contradiction to our assumption that $(X_0, \dots, X_n)$ is optimal. If $p = 0$, $q \geq 2$ we would obtain a diagram

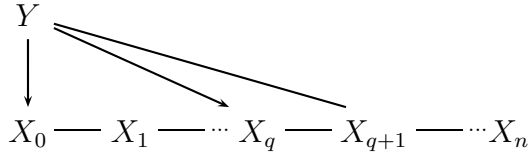
\[
\setlength{\arraycolsep}{0.05cm}
\begin{array}{cccc}
\Rnode{a}{Y} &    &      &     \\[1cm]
\Rnode{b}{X_0} & \Rnode{c}{^{\hbox to 0.6cm{\hrulefill}} \; X_1 \; ^{\hbox to 0.6cm{\hrulefill} \; \dots}} & \Rnode{d}{X_q \; ^{\hbox to 0.6cm{\hrulefill}}} & \Rnode{e}{X_{q+1} \;  ^{\hbox to 0.6cm{\hrulefill} \; \dots} X_n}\\[1mm]
\end{array}
\psset{nodesep=5pt,arrows=->} \everypsbox{\scriptstyle}
\ncline{a}{b} \ncline{a}{d}
\psset{nodesep=5pt,arrows=-} \everypsbox{\scriptstyle}
\ncline{a}{e}
\]
and a QL-path $(X_0, Y, X_{q+1}, \dots, X_n)$ which is an $f$-elementary reduction of $(X_0, \dots, X_n)$ in contradiction to our assumption that $(X_0, \dots, X_n)$ is optimal. In the same way we see that $p \leq n-2$, $q = n$ is impossible. Thus any two of the sets listed in the theorem have empty intersection. \end{proof}

\begin{cor}\label{cor:15.2}
 Assume that $(Z_0, \dots, Z_n)$ is an enlargement of an optimal QL-path. Then the downset $Y^{\downarrow}$ of any ray $Y$ in $V$ meets each of the three sets $\{ Z_1, \dots, Z_{n-1} \}$, $\{ Z_0, Z_2, \dots, Z_n \}$, $\{ Z_0, \dots, Z_{n-2}, Z_n \}$ in at most one ray.
\end{cor}

\begin{proof} If $Y^{\downarrow}$ contains two rays $Z_i, Z_j$ $(0 \leq i < j \leq n)$, then $Y \in Z_i^{\uparrow} \cap Z_j^{\uparrow}$.
This is excluded by Theorem~\ref{thm:15.1}, if $Z_i$ and $Z_j$ are elements of one of these three sets. \end{proof}

The set of enlargements of a given QL-path $(X_0, \dots, X_n)$ can be very rich, as is indicated by the following fact.

\begin{prop}\label{prop:15.3}
 Assume that $(Y_0, \dots, Y_n)$ is an enlargement of $(X_0, \dots, X_n)$. Then every sequence of rays $(Z_0, \dots, Z_n)$ with $Z_i \in [X_i, Y_i]$ for $0 \leq i \leq n$ is again an enlargement of $(X_0, \dots, X_n)$ (but often $(Y_0, \dots, Y_n)$ is not an enlargement of $(Z_0, \dots, Z_n)$).\end{prop}

\begin{proof} $Z_i \in X_i^{\uparrow}$, since by Theorem \ref{thm:4.6} each set $X_i^{\uparrow}$ is convex. Furthermore, each pair $(Z_i, Z_{i+1})$, $0 \leq i < n$, is quasilinear, since the convex hull of $\{ X_i, X_{i+1}, Y_i, Y_{i+1} \}$ is convex, as we know for long. \end{proof}

\begin{thm}\label{thm:15.4}
 Assume that $(X_0, \dots, X_n)$ is a \textbf{direct} QL-path and $Y$ is a ray in $V$ with $Y^{\uparrow} \cap \{ X_0, \dots, X_n \} \ne \emptyset$. Then $Y^{\uparrow}$ meets the set $\{ X_0, \dots, X_n \}$ either in exactly one ray $X_p$ or in exactly two rays $X_p, X_{p+1}$. In the first case $X_i \not\in \QL (Y)$ for $\vert i - p \vert > 1$ if $0 < p < n$, while $X_i \not\in \QL (Y)$ for $i \geq 2$ if $ p = 0$, and $X_i \not\in \QL (Y)$ for $i \leq n-2$ if $p = n$. In the second case $X_i \not\in \QL (Y)$ for $i \not\in \{ p, p+1 \}$.
\end{thm}
\begin{proof} All assertions are immediate consequences of the following three observations.
\begin{itemize}
\item[a)] Assume that $X_p \in Y^{\uparrow}$ for some $p \in [0, n]$ and that $X_i \in \QL (Y)$ for some $i \in [0, n]$. Then $X_i \in \QL (X_p)$ because $\QL (Y) \subset \QL (X_p)$. Since the path $(X_0, \dots, X_n)$ is direct, this implies $\vert i - p \vert \leq 1$. Thus if $\vert i - p \vert > 1$, then $X_i \not\in \QL (Y)$ and all the more $\QL (X_i) \not\supset \QL (Y)$, i.e. $X_i \not\in Y^{\uparrow}$.
\item[b)] Assume that $X_p \in Y^{\uparrow}$ and $X_q \in Y^{\uparrow}$ for two indices $p < q$ in $[0, n]$, and that $X_i \in \QL (Y)$ for some $i \in [0, n]$. Then $X_i \in \QL (X_p)$ and $X_i \in \QL (X_q)$ because $\QL (Y) \subset \QL (X_p) \cap \QL (X_q)$. Since $(X_0, \dots, X_n)$ is direct this forces $\vert i - p \vert \leq 1$ and $\vert i - q \vert \leq 1$. Thus, if $\vert i - p \vert > 1$ or $\vert i - q \vert > 1$ then $X_i \not\in \QL (Y)$, and all the more $\QL (Y) \not\subset \QL (X_i)$, i.e. $X_i \not\in Y^{\uparrow}$.
\item[c)] Given a ray $X_p \in Y^{\uparrow}$ we conclude from a) that $X_i \in Y^{\uparrow}$ at most for $i = p-1$, $p, p+1$ if $0 < p < n$, while for $p = 0$ $X_i \in Y^{\uparrow}$ at most for $i = 0,1$ and for $p = 0$ at most for $i = n-1, n$. Thus $Y^{\uparrow}$ contains either none or one or two adjacent rays in $\{ X_0, \dots, X_n \}$.
\end{itemize} \vskip -5mm
\end{proof}

%\begin{defn}\label{thm:15.5}
% If $(X_0, \dots, X_n)$ is a direct path of length $n \geq 2$, then we call two rays $X_p, X_{p+1}$ $(0 \leq p < n)$ with $X_p^{\downarrow} \cap X_{p+1}^{\downarrow} \ne \emptyset$ \textbf{twins}, all other rays \textbf{singles} in the path $(X_0, \dots, X_n)$.
%\end{defn}
%

\begin{defn}\label{def:15.5} $ $
\begin{enumerate} \eroman
\item Suppose  $(X_0, \dots, X_n)$ is a direct $\QL$-path of length $n \geq 1$. We call two adjacent rays $X_p, X_{p+1}$ $(0 \leq p < n)$ \textbf{twins}, if $X_p^{\downarrow} \cap X_{p+1}^{\downarrow} \ne \emptyset$, and also say that $(X_p, X_{p+1})$ is a \textbf{twin pair}.  We say that a ray $X_q$ $(0 \leq q \leq n)$ is a \textbf{single} in the $\QL$-path $(X_0, \dots, X_n)$, if $X_q$ is not a twin, i.e., it is not part of a twin pair.

\item We call a ray $Y$ an \textbf{anchor} of a ray $X_i$ in $(X_0, \dots, X_n)$, if either $X_i$ has a twin $X_j$ $(\vert j - i \vert = 1)$ and $Y \in X_i^{\downarrow} \cap X_j^{\downarrow}$, or $X_i$ is a single and $Y \in X_i^{\downarrow}$. We then also say that $X_i$ is \textbf{anchored at} $Y$.
\end{enumerate}
\end{defn}

By Theorem~\ref{thm:15.4}, a path is single  iff $X_q^{\downarrow} \cap X_j^{\downarrow} = \emptyset$ for every $j$ with $\vert j-q \vert = 1$, which then holds for all $j \ne q$ in $\{ 0, \dots, n \}$.

\begin{proc}\label{proc:15.5}

 We choose an ordered set $S = (Y_0, \dots, Y_m)$ of anchors for the direct QL-path $T = (X_0, \dots, X_n)$ such that for each $i \in [0, n]$ the set of chosen anchors $(Y_0, \dots, Y_k)$ for $(X_0, \dots, X_i)$ is as small as possible, and then call $S$ an \textbf{anchor set of} $T$. More precisely we proceed as follows. We choose an anchor $Y_0$ of $X_0$. If $(X_0, X_1)$ is a twin pair, we choose the  anchor $Y_0$ also for $X_1$. Otherwise we choose for $X_1$ a new anchor $Y_1$. If anchors $(Y_0, \dots, Y_k)$ have been chosen for $(X_0, \dots, X_i)$, $i < n$, we choose for $X_{i+1}$ again the anchor $Y_k$, if $(X_i, X_{i+1})$ is a twin pair and $Y_k$ has not already been chosen twice in the anchor list for $(X_0, \dots, X_i)$, which means that $(X_{i-1}, X_i)$ is not a twin pair.\footnote{cf. \S\ref{sec:16} below about different twin-pairs which are not disjoint.} Otherwise we choose for $X_{i+1}$ a new anchor $Y_{k+1}$. Note that for different rays $X_i, X_j$ with $i < j$ we have in $S$ anchors $Y_k, Y_{\ell}$ with $k \leq \ell$.

 In particular  we can choose as anchor of a single $X_q$ the ray $X_q$ itself. An anchor set $S$ arising in this way is called a \textbf{special anchor set} of the direct $\QL$-path $(X_0, \dots, X_n)$.
\end{proc}

\begin{rem}\label{rem:15.6} $ $
\begin{itemize}
\item[a)] All anchor sets of $T =(X_0, \dots, X_n)$ have the same length $m$. Moreover $m \leq n \leq 2 m$, where  $n = m$ if $S$ contains only singles and $n = 2m$ if $T$ contains only twins, where no two twin pairs have a ray in common.
\item[b)] If $(Y_0, \dots, Y_m)$ is an anchor set of $(X_0, \dots, X_n)$, then every tuple of rays $(Y'_0, \dots, Y'_m)$ with $Y'_k \leq_{\QL} Y_k$ for $0 \leq k \leq m$ is again an anchor set of $(X_0, \dots, X_n)$.
\end{itemize}
All this is evident from Definition \ref{def:15.5}.
\end{rem}

We turn to the problem of  specifying  which finite ordered sets $(Y_0, \dots, Y_m)$ of rays in ~$V$ $(m \geq 1)$ can serve as an anchor set of a direct $\QL$-path. The following fact will be of help.

\begin{lem}\label{lem:15.7}
 Assume that $X_1, X_2, Y_1, Y_2$ are rays in $V$ with $Y_1 \in X_1^{\downarrow}$, $Y_2 \in X_2^{\downarrow}$, and  that the pair $(Y_1, Y_2)$ is quasilinear. Then $(X_1, X_2)$ is quasilinear.
\end{lem}
\begin{proof}
 We have $\QL (Y_1) \subset \QL (X_1)$, $\QL (Y_2) \subset \QL (X_2)$, and $Y_2 \in \QL (Y_1)$. From this we conclude that $Y_2 \in \QL (X_1)$, and then that $X_1 \in \QL (Y_2) \subset \QL (X_2)$, which proves that $(X_1, X_2)$ is quasilinear. \end{proof}

\begin{defn}\label{def:15.8}
 We call a pair $(Y_1, Y_2)$ of rays in $V$ \textbf{subquasilinear}, abbreviated \textbf{sql}, if there exists a quasilinear pair $(X_1, X_2)$ with  $Y_1 \in X_1^{\downarrow}$, $Y_2 \in X_2^{\downarrow}$, and say that the $\QL$-pair $(X_1, X_2)$ \textbf{covers} $(Y_1, Y_2)$. We  call a tuple of rays $(Y_0,  \dots, Y_m)$, $m \geq 1$, \textbf{subquasilinear}, if $(Y_i, Y_{i+1})$ is sql for $0 \leq i < m$, and say  that $(Y_0,  \dots, Y_m)$ is a \textbf{subquasilinear sequence} of \textbf{length} $m$. Finally, we say that an sql sequence $(Y_0, \dots, Y_m)$ is \textbf{direct}, if for any rays $Y_k, Y_{\ell}$ with $\vert k - \ell \vert > 1$ in the sequence  the pair $(Y_k, Y_{\ell})$ is \textbf{not}  quasilinear.
\end{defn}

\begin{thm}\label{thm:15.9}
Every anchor set $S = \{ Y_0, \dots, Y_m \}$ of a direct quasilinear path $(X_0, \dots, X_n)$ is a direct subquasilinear sequence.
\end{thm}
\begin{proof}
 Let $Y_k$ and $Y_{\ell}$ be different rays in $S$ with $k < \ell$, and let $X_i$ and $X_j$ be rays in $(X_0, \dots, X_n)$ which are anchored at $Y_k$ and $Y_{\ell}$ respectively, $i < j$. Thus $Y_k \in X_i^{\downarrow}$ and $Y_{\ell} \in X_j^{\downarrow}$. If $j > i + 1$, then $(X_i, X_j)$ is not quasilinear, since the $\QL$-path $(X_0, \dots, X_n)$ is direct. It follows by Lemma~\ref{lem:15.7} that $(Y_k, Y_{\ell})$ is \textbf{not} quasilinear. Assume now that $j = i + 1$. Then it is clear by Procedure~\ref{proc:15.5} that $\ell = k + 1$. Since the $\QL$-pair $(X_i, X_{i+1})$ covers that pair $(Y_k, Y_{k+1})$, evidently, $(Y_0, \dots, Y_m)$ is a direct subquasilinear sequence. \end{proof}

\section{Minimal QL-paths and their anchor sets; the appearance of flocks }\label{sec:16}

Our first topic in this section is the case that in a direct QL-path there exist adjacent twin pairs, which are not disjoint.

\begin{defn}\label{def:16.1}
 We call a subsequence $(X_p, \dots, X_q)$ of a direct QL-path $T = (X_0, \dots, X_n)$ with $q - p \geq 2$ a \textbf{flock} in $T$, if $(X_i, X_{i+1})$ is a twin pair in $T$ for every $i$ in $[p, q-1]$, and so $T$ has an anchor set $S = (Y_0, \dots, Y_m)$, in which these twin pairs have anchors $Y_k, Y_{k+1}, \dots, Y_{k+t}$ for some $k \in [0, m]$ and $t = q - p -1$. In other terms, we have a diagram

\begin{equation}\label{eq:16.1}
\setlength{\arraycolsep}{0.1cm}
\begin{array}{ccccccccc}
\Rnode{a}{X_p} & \Rnode{b}{^{\hbox to 0.6cm{\hrulefill}}} & \Rnode{c}{X_{p+1}} & \Rnode{d}{^{\hbox to 0.6cm{\hrulefill}}}
 & \Rnode{e}{X_{p+2}} & \Rnode{f}{\raisebox{1mm}[4mm][2mm]{\hbox to 0.6cm{\hrulefill} \dots \hbox to 0.6cm{\hrulefill}}} & \Rnode{g}{X_{q-1}} & \Rnode{h}{^{\hbox to 0.6cm{\hrulefill}}} & \Rnode{i}{X_q} \\[1cm]
               & \Rnode{j}{Y_k} &      & \Rnode{k}{Y_{k+1}} &      & \Rnode{l}{Y_{k+2}} &      &
\Rnode{m}{Y_{k+t}} & \qquad .     \\
\end{array}
\psset{nodesep=5pt,arrows=->} \everypsbox{\scriptstyle}
\ncLine{a}{j} \ncLine{c}{j}
\ncline{c}{k} \ncline{e}{k}
\ncLine{e}{l} \ncLine{g}{m}
\ncline{i}{m}
\end{equation}
We call $q-p = t+1$ the \textbf{length} of the flock $(X_p, \dots, X_q)$. We further call a twin pair in $T$, which is not a member of a flock, an \textbf{isolated twin pair}.
\end{defn}

\begin{rem}\label{rem:16.2} $ $
\begin{itemize}
\item[a)] In a flock, as seen in diagram \eqref{eq:16.1}, the common ray of any two adjacent twin pairs has two anchors in the sequence $S$, while each ray in an isolated twin pair has only one anchor in $S$.
\item[b)] None of the adjacent pairs $(Y_{k+i-1}, Y_{k+i})$ in \eqref{eq:16.1} is quasilinear (but, of course, is ~sql), since otherwise the QL-path $(X_p, \dots, X_q)$ would not be direct. Indeed, if, say, $(Y_k, Y_{k+1})$ would be ql, then by Lemma~\ref{lem:15.7} the pair $(X_p, X_{p+2})$ would be ql.
\item[c)] For the same reason it cannot happen in $S$, that two non-adjacent rays form a quasilinear pair, as already stated in Theorem~\ref{thm:15.9}.
\end{itemize}
\end{rem}

\begin{lem}\label{lem:16.3}
 Assume that $T = (X_0, \dots, X_n)$ is a QL-path, and that for given $p, q \in [0, n]$ with $p+1 < q$ the sets $(X^{\downarrow}_p)^{\uparrow}$ and $(X^{\downarrow}_q)^{\uparrow}$ are not disjoint. We choose $Y_1 \in X_p^{\downarrow}$, $Y_2 \in X_q^{\downarrow}$ such that there is some $Z \in Y_1^{\uparrow} \cap Y_2^{\uparrow}$.

\begin{itemize}
  \item[a)] Then we have a QL-path $(X_0, \dots, X_p, Z, X_q, \dots, X_n)$ with a diagram
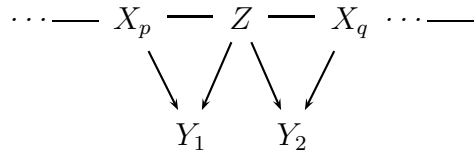
\[
\setlength{\arraycolsep}{0.1cm}
\begin{array}{ccccccc}
\Rnode{a}{\raisebox{1mm}[4mm][2mm]{\dots \hbox to 0.6cm{\hrulefill}}} & \Rnode{b}{X_p} & \Rnode{c}{^{\hbox to 0.6cm{\hrulefill}}} & \Rnode{d}{Z}
 & \Rnode{e}{^{\hbox to 0.6cm{\hrulefill}}} & \Rnode{f}{X_q} & \Rnode{g}{\raisebox{1mm}[4mm][2mm]{\dots \hbox to 0.6cm{\hrulefill}}} \\[1cm]
    &      & \Rnode{h}{Y_1} &      & \Rnode{i}{Y_2} &      &    \\
\end{array}
\psset{nodesep=5pt,arrows=->} \everypsbox{\scriptstyle}
\ncLine{b}{h} \ncLine{d}{h}
\ncline{d}{i} \ncline{f}{i}
\]

\item[b)] If the new path $T' = (X_0, \dots, X_p, Z, X_q, \dots, X_n)$ is direct, then $(X_p, Z)$ and $(Z, X_q)$ are twin pairs, and so $(X_p, Z, X_q)$ is a flock of length 2 in $T'$.
\item[c)] If $q = p + 2$, then $T'$ has again length $n$; otherwise $T'$ is shorter.

\end{itemize}
\end{lem}
\begin{proof}
 The pairs $(X_p, Z)$ and $(Z, X_q)$ are quasilinear, since $Y_1^{\uparrow}$ and $Y_2^{\uparrow}$ are convex quasilinear sets (Theorem~\ref{thm:4.6}). Thus $T'$ is indeed a quasilinear path. Now assertions b) and c) are immediate by Definition~\ref{def:16.1} and an easy counting. \end{proof}

The crux in this lemma is that, even if we assume that $T$ is direct, in general there is no apparent way to decide whether $T'$ is direct or not.

\begin{defn}\label{def:16.4}
 We call a QL-path $(X_0, \dots, X_n)$ \textbf{minimal}, if there does not exist a QL-path from $X_0$ to $X_n$ of length $< n$.
\end{defn}

If in Lemma~16.3 the path $T$ is minimal, then $T'$ is again minimal (and $q = p+2$), and so~ $T'$ is certainly direct.

\begin{thm}\label{thm:16.5}
 Assume that $T = (X_0, \dots, X_n)$ is a minimal QL-path of length $n \geq 2$, and $S = (Y_0, \dots, Y_m)$ is an anchor set of $T$.
Assume further that $(Y_k, \dots, Y_{k+t+1})$ is a maximal subsequence of $S$ with $t \geq 0$, $k \geq 0$, and $k + t < m$, such that
\begin{equation}\label{eq:16.2}
 Y_i^{\uparrow} \cap Y_{i+1}^{\uparrow} \ne \emptyset \quad \mbox{for} \; k \leq i \leq k + t. \end{equation}
Choose rays $X_s$ and $X_r$ in $T$ such that $Y_k$ is an anchor of $X_s$ and $Y_{k+t+1}$ an anchor of $X_r$ in~ $S$. Finally choose rays $W_i \in Y_i^{\uparrow} \cap Y_{i+1}^{\uparrow}$ for $k \leq i \leq k+t$. Then
\[ T' := (X_0, \dots, X_s, W_k, \dots, W_{k+t}, X_r, \dots, X_n) \]

\noi
is again a minimal QL-path of length $n$ which admits $S$ as anchor set. The sequence $(X_s, W_k, \dots, W_{k+t}, X_r)$ is a flock in $T$ of length $t + 2$ with anchors $Y_k, \dots, Y_{k+t+1}$. This flock is maximal in $T'$, i.e., there is no flock of length $> t + 2$ in $T'$ which contains $(X_s, W_k, \dots, W_{k+t}, X_r)$.
\end{thm}
\begin{proof}
  Using  Lemma \ref{lem:16.3} iteratively, we obtain a diagram
\begin{equation}\label{eq:str}
\setlength{\arraycolsep}{0.1cm}
\begin{array}{ccccccccccccc}
\Rnode{a}{^{\dots} \; \longrightarrow} &\Rnode{b}{X_s} & \Rnode{c}{^{\hbox to 0.6cm{\hrulefill}}} & \Rnode{d}{W_k} & \Rnode{e}{^{\hbox to 0.6cm{\hrulefill}}} & \Rnode{f}{X_{s+1}} & \Rnode{g}{^{\hbox to 0.6cm{\hrulefill}}} & \Rnode{h}{W_{k+1}} & \Rnode{i}{^{\hbox to 0.6cm{\hrulefill}}} & \Rnode{j}{X_{s+2}} & \Rnode{k}{\longrightarrow} & \Rnode{l}{W_{k+2}} & \Rnode{m}{^{\hbox to 0.6cm{\hrulefill} \; \dots}} \\[1cm]
    &    & \Rnode{n}{Y_k} &   &   & \Rnode{o}{Y_{k+1}} &   &    &   & \Rnode{p}{Y_{k+2}} &    &    &    \\
\end{array}
\psset{nodesep=5pt,arrows=->} \everypsbox{\scriptstyle}
\ncLine{b}{n} \ncLine{d}{n}
\ncline{d}{o} \ncline{f}{o}
\ncLine{h}{o} \ncLine{h}{p}
\ncline{j}{p} \ncline{l}{p}
\tag{$*$}\end{equation}
The upper row is a QL-path, in which the pairs $(W_k, W_{k+1})$, $(W_{k+1}, W_{k+2}), \dots$ are quasilinear, since they are contained in the convex quasilinear sets $Y_{k+1}^{\uparrow}$, $Y_{k+2}^{\uparrow}, \dots$. Thus ($\ast$) can be reduced to a diagram
\begin{equation}\label{eq:16.3}
\setlength{\arraycolsep}{0.1cm}
\begin{array}{ccccccccccc}
\Rnode{a}{^{\dots \; \hbox to 0.6cm{\hrulefill}}} & \Rnode{b}{X_s} & \Rnode{c}{^{\hbox to 0.6cm{\hrulefill}}} & \Rnode{d}{W_k} & \Rnode{e}{^{\hbox to 0.6cm{\hrulefill}}} & \Rnode{f}{W_{k+1}} & \Rnode{g}{^{\hbox to 0.6cm{\hrulefill} \dots}} & \Rnode{h}{W_{k+t}} & \Rnode{i}{^{\hbox to 0.6cm{\hrulefill}}} & \Rnode{j}{X_r} & \Rnode{k}{^{\hbox to 0.6cm{\hrulefill} \; \dots}} \\[1cm]
    &    & \Rnode{l}{Y_k} &      & \Rnode{m}{Y_{k+1}} &      & \Rnode{n}{Y_{k+t}} &      &
\Rnode{o}{Y_{k+t+1}} &  & \qquad ,     \\
\end{array}
\psset{nodesep=5pt,arrows=->} \everypsbox{\scriptstyle}
\ncLine{b}{l} \ncLine{d}{l}
\ncline{d}{m} \ncline{f}{m}
\ncLine{h}{n} \ncLine{h}{o}
\ncline{j}{o}
\end{equation}
where to the left of $W_k$ and $Y_k$, and to the right of $W_{k+t}$ and $Y_{k+t}$ there appear parts of the anchor diagram of $T$ over $S$. In the upper row of \eqref{eq:16.3} we have a QL-path
$$ T':= (X_0, \dots, X_s, W_k, \dots, W_{k+t}, X_r, \dots, X_n) $$
of length $\leq n$ starting at $X_0$ and ending at $X_n$, as does $T$. Thus $T'$ is minimal, whence direct, so that we can speak about twin pairs and flocks in $T'$, and $T'$ has again length $n$.

From diagram \eqref{eq:16.3} we read off that $(X_s, W_k, \dots, W_{k+t}, X_r)$ is a flock in $T'$ of length $t + 2$. Since $Y_{k-1}^{\uparrow} \cap Y_k^{\uparrow} = \emptyset$ and $Y_{k+t+1}^{\uparrow} \cap Y_{k+t+2}^{\uparrow} = \emptyset$, it is evident that $Y_{k-1}$ and $Y_{k+t+2}$ are not anchors of $X_s$ and $X_r$ respectively. It follows that $T'$ admits $S$ as an anchor set,\footnote{Recall Procedure~\ref{proc:15.5}.} and then, that the flock $(X_s, W_k, \dots, W_{k+t}, X_r)$ is maximal in $T'$. \end{proof}

Exploring properties of anchor sets beyond Theorem~\ref{thm:15.9}, we take a closer look at Procedure~\ref{proc:15.5}  to obtain such sets.

\begin{defn}\label{def:16.6}
 Assume that $T = (X_0, \dots, X_n)$ is a direct QL-path and $S = (Y_0, \dots, Y_n)$ is  an anchor set of $T$. If $X_i$ is a ray in $T$, we call the ray $Y_k$ chosen for $X_i$ as anchor in $S$ the \textbf{legal anchor} of $X_i$ in $S$. The other anchors of $X_i$ in $S$ are named \textbf{illegal anchors}. \end{defn}
\noi (In fact $X_i$ can have at most one illegal anchor in $S$, see below).

Looking at Procedure~\ref{proc:15.5}, the following is now easily verified.

\begin{schol}\label{schol:16.7}

$ $
\begin{enumerate}
  \item[ a)] A single $X_i$ in $T$ has only a legal anchor in $S$.

\item[b)] Both rays in an isolated twin pair $(X_i, X_{i+1})$ have one common legal anchor in $S$, and no illegal anchors.
\item[c)] If $(X_p, X_{p+1}, \dots, X_q)$ is a maximal flock in $T$, $q = p + t +1$ with $t \geq 0$, then each \textit{interior ray} $X_i$, $p < i < q$, of the flock has one legal and one illegal anchor in $S$, while the \textit{border rays} $X_p$ and $X_q$ have only legal anchors. These are $Y_k$ and $Y_{k+t+1}$. The legal anchors of $X_{p+1}, \dots, X_{p+t}$ are $Y_k, \dots, Y_{k+t}$, while the illegal anchors of these rays are $Y_{k+1}, \dots, Y_{k+t+1}$.
\item[d)] It follows that, if a ray $X_i$ has a legal anchor $Y_u$ and an illegal anchor $Y_w$ in $S$, then $w = u+1$.
    \end{enumerate}

    \end{schol}

\begin{thm}\label{thm:16.8}
 Assume that $S = (Y_0, \dots, Y_m)$ is an anchor set of a direct QL-path $T = (X_0, \dots, X_n)$. Then the downsets $Y_0^{\downarrow}, \dots, Y_m^{\downarrow}$ are pairwise disjoint.
\end{thm}
\begin{proof}
 Given rays $Y_k$ and $Y_{\ell}$ with $k < \ell$, suppose that there exists a ray $Z \in Y_k^{\downarrow} \cap Y_{\ell}^{\downarrow}$. We pick rays $X_p$ and $X_q$ in $T$, with legal anchors $Y_k$ and $Y_{\ell}$, respectively. Then $p < q$ and
\[ X_p \in Y_k^{\uparrow} \subset Z^{\uparrow}  , \quad X_q \in Y_{\ell}^{\uparrow} \subset Z^{\uparrow}. \]

\noi
Thus $Z^{\uparrow}$ meets $T$ in two rays $X_p$ and $X_q$. So $q = p+1$ and $(X_p, X_q)$ is a twin-pair with common anchor $Z$. This forces $\ell = k+1$. We infer from Scholium~\ref{schol:16.7}.b that the twin pair $(X_p, X_{p+1})$ is \textit{not isolated} in $T$, and thus is part of a maximal flock in $T$ which either extends to the left of $X_p$ or to the right of $X_{p+1}$ (or both). Thus we have a diagram

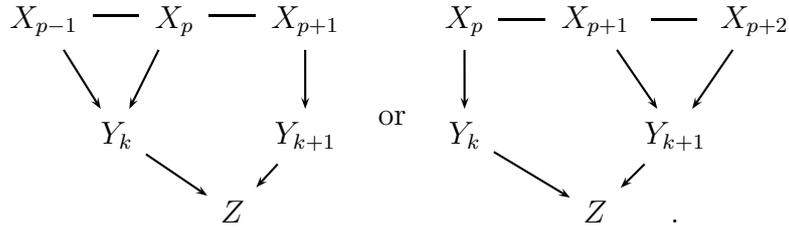
\[
\setlength{\arraycolsep}{0.1cm}
\begin{array}{ccccc}
\Rnode{a}{X_{p-1}} & \Rnode{b}{^{\hbox to 0.6cm{\hrulefill}}} & \Rnode{c}{X_p} & \Rnode{d}{^{\hbox to 0.6cm{\hrulefill}}} & \Rnode{e}{X_{p+1}}  \\[1cm]
    & \Rnode{f}{Y_k} &  &  & \Rnode{g}{Y_{k+1}}  \\[0.5cm]
    &  &  & \Rnode{h}{Z} & \\
\end{array}
\psset{nodesep=5pt,arrows=->} \everypsbox{\scriptstyle}
\ncLine{a}{f} \ncLine{c}{f}
\ncline{e}{g} \ncline{f}{h}
\ncLine{g}{h}
\quad \mbox{or} \quad
\setlength{\arraycolsep}{0.1cm}
\begin{array}{ccccc}
\Rnode{a}{\displaystyle X_p} & \Rnode{b}{^{\hbox to 0.6cm{\hrulefill}}} & \Rnode{c}{\displaystyle X_{p+1}} & \Rnode{d}{^{\hbox to 0.6cm{\hrulefill}}} & \Rnode{e}{\displaystyle X_{p+2}}  \\[1cm]
\Rnode{f}{\displaystyle Y_k} & & & \Rnode{g}{\displaystyle Y_{k+1}} & \\[0.5cm]
    &  & \Rnode{h}{\displaystyle Z} & . & \\
\end{array}
\psset{nodesep=5pt,arrows=->} \everypsbox{\scriptstyle}
\ncLine{a}{f} \ncLine{c}{g}
\ncline{e}{g} \ncline{f}{h}
\ncLine{g}{h}
\]

\noi
Both diagrams cannot exist, since then $Z^{\uparrow}$ would contain at least three rays of $T$, contradicting Theorem~\ref{thm:15.9}. Thus $Y_k^{\downarrow}$ and $Y_{\ell}^{\downarrow}$ are disjoint. \end{proof}

In a similar way we obtain a result about the upsets $Y_k^{\uparrow}$.

\begin{thm}\label{thm:16.9}
 Assume again that $S = (Y_0, \dots, Y_m)$ is an anchor set of a direct QL-path $T = (X_0, \dots, X_n)$. Given rays $Y_k, Y_{\ell}$ in $S$ with $k < \ell$, the intersection $Y_k^{\uparrow} \cap Y_{\ell}^{\uparrow}$ contains at most one ray of $T$, and in this case $\ell = k+1$.
\end{thm}
\begin{proof}
 Suppose that there exist two rays $X_p, X_q, p < q$, in $S$ which both are in $Y_k^{\uparrow} \cap Y_{\ell}^{\uparrow}$. Since $Y_k^{\uparrow}$ meets $S$ in two rays $X_p, X_q$ it is evident (cf. Theorem~\ref{thm:15.4}), that $q = p+1$ and $(X_p, X_q)$ is a twin-pair. This forces $\ell = k+1$. Both $Y_k$ and $Y_{k+1}$ are common anchors of $X_p$ and $X_{p+1}$ in $S$. We conclude from Scholium~\ref{schol:16.7}, that $Y_k$ is the legal anchor of both $X_p$ and $X_{p+1}$, while $Y_{k+1}$ is an illegal anchor of both $X_p$ and $X_{p+1}$, and then, that $X_p$ and $X_{p+1}$ are interior rays of a maximal flock in $T$. Thus we have a diagram
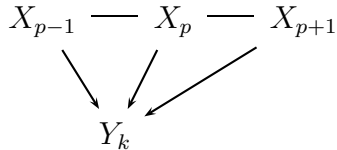
\[
\setlength{\arraycolsep}{0.1cm}
\begin{array}{ccccc}
\Rnode{a}{X_{p-1}} & \Rnode{b}{^{\hbox to 0.6cm{\hrulefill}}} & \Rnode{c}{X_p} & \Rnode{d}{^{\hbox to 0.6cm{\hrulefill}}} & \Rnode{e}{X_{p+1}}  \\[1cm]
  & \Rnode{f}{Y_k} & & & \\
\end{array}
\psset{nodesep=5pt,arrows=->} \everypsbox{\scriptstyle}
\ncLine{a}{f} \ncLine{c}{f}
\ncline{e}{f}
\]
and an analogous diagram involving $X_p, X_{p+1}, X_{p+2}, Y_{k+1}$. But such diagrams do not exist, since $Y_k^{\uparrow}$, and also $Y_{k+1}^{\uparrow}$, can meet $T$ in at most 2 rays (cf. Theorem~\ref{thm:15.4}). We conclude that $Y_k^{\uparrow} \cap Y_{k+1}^{\uparrow}$ contains at most one ray of $T$. \end{proof}

\begin{defn}\label{def:16.10}
 We call a pair $(T, S)$ consisting of a direct QL-path $T = (X_0, \dots, X_n)$ and an anchor set $S = (Y_0, \dots, Y_m)$ of $T$ \textbf{flocky}, if for any pair $(Y_i, Y_{i+1})$ in $S$ with $Y_i^{\uparrow} \cap Y_{i+1}^{\uparrow} \ne \emptyset$ there exists a ray $W$ in $T$ with $W \in Y_i^{\uparrow} \cap Y_{i+1}^{\uparrow}$. \end{defn}
\noindent Note that then $W$ is the unique such ray, as stated in Theorem~\ref{thm:16.9}.

We have an explicit description of all maximal flocks in a flocky pair as follows.

\begin{thm}\label{thm:16.11}
 Assume that $(T, S)$ is a flocky pair, $T = (X_0, \dots, X_n)$, $S = (Y_0, \dots, Y_m)$. Assume further that $(Y_k, \dots, Y_{k+t+1})$ is a maximal subsequence of $S$ with
\begin{equation}\label{eq:16.4}
 t \geq 0, k \geq 0, \; k + t < m, \; Y_i^{\uparrow} \cap Y_{i+1}^{\uparrow}\ne \emptyset \quad \mbox{for} \quad k \leq i \leq k + t. \end{equation}
Let $W_i$ denote the unique ray of $T$ contained in $Y_i^{\uparrow} \cap Y_{i+1}^{\uparrow} \; (k \leq i \leq k + t)$. We have $W_k = X_p$, $W_{k+t} = X_q$ with indices $p < q$ in $[0, n]$.
\begin{itemize}
\item[a)] The sequence $(W_k, \dots, W_{k+t})$ coincides with the subpath $(X_p, \dots, X_q)$ of $T$, in particular $q = p + t$, and this is a maximal subpath of $T$ with the property that each of its rays has two anchors in $S$. These are the legal anchors $Y_k, \dots, Y_{k+t}$ and the illegal anchors $Y_{k+1}, \dots, Y_{k+t+1}$. The subpath $(X_p, \dots, X_q)$ is a flock in $T$, perhaps not a maximal flock.
\item[b)] Assume that $p > 0$ and $q = p + t < n$. Then the unique anchor of $X_{p-1}$ in $S$ is either $Y_k$ or $Y_{k-1}$, and the unique anchor of $Y_{q+1}$ in $S$ is either $Y_{k+t}$ or $Y_{k+t+1}$. If $ X_{p-1}$ and $Y_{q+1}$ have the anchors $Y_k$ and $Y_{k+t}$, then the maximal flock containing $(X_p, \dots, X_q)$ is $(X_{p-1}, \dots, X_{q+1})$ while in the other cases we have to delete either $X_{p-1}$ or $X_{q+1}$ or both in the subpath $(X_{p-1}, \dots, X_{q+1})$ to obtain a maximal flock.
\item[c)] In the case $p = 0$, $q < n$, we have the maximal flock $(X_0, \dots, X_{q+1})$ if $X_{q+1}$ has anchor $Y_{k+t}$, and $(X_0, \dots, X_q)$ if $X_{q+1}$ has anchor $Y_{k+t+1}$. In the case $p > 0$, $q = n$ we have the maximal flock $(X_{p-1}, \dots, X_n)$ if $X_{p-1}$ has anchor $Y_k$, and $(X_p, \dots, X_n)$ if $X_{p-1}$ has anchor $Y_{k-1}$. In the trivial case $p = 0$, $q = n$, the path $T$ itself is a flock in $T$.
\end{itemize}
\end{thm}
\begin{proof}
 We focus on the case $p > 0$, $q < n$. Every pair $(W_i, W_{i+1})$, $k \leq i \leq k + t$ is quasilinear since both $W_i$ and $W_{i+1}$ are rays in the convex quasilinear set $Y_{i+1}^{\uparrow}$. Thus $(W_k, \dots, W_{k+t})$ is a QL-path. All $W_i$ are rays in the QL-path $(X_p, \dots, X_q)$. Since this QL-path is direct, a pair $(X_i, X_j)$ with $p \leq i < j \leq q$ can be quasilinear only if $j = i + 1$. This forces
\[ (W_k, \dots, W_{k+t}) = (X_p, \dots, X_q) \]

\noi
and $q = p + t$. Obviously each $W_i$ has in $S$ two anchors, the legal anchor $Y_i$ and the illegal anchor $Y_{i+1}$, and $(W_k, \dots, W_{k+t})$ is a flock in $T$. $X_{p-1}$ cannot have two anchors in $S$, since these would be $Y_{k-1}$ and $Y_k$ (cf. Scholium \ref{schol:16.7}), contradicting the maximality of the family $(Y_k, \dots, Y_{k+t+1})$ with \eqref{eq:16.4}. Same for $X_{q+1}$. Thus the path $(X_p, \dots, X_q)$ is maximal in $T$ with the property, that each of its rays has two anchors in $S$. If $X_{p-1}$ has the anchor $Y_k$ then $(X_{p-1}, X_p)$ is a twin pair with anchor $Y_k$. If $X_{p-1}$ has the (only) anchor $Y_{k-1}$, then $(X_{p-1}, X_p)$ is not a twin pair. Thus is the first case $(X_{p-1}, \dots, X_q)$ is a flock, but in the second case not. Analogously $(X_p, \dots, X_{q+1})$ is a flock if and only if $X_{q+1}$ has the anchor ~$Y_{k+t}$. This proves all claims in the case $p > 0$, $q < n$. In the cases $p = 0$, $q < n$ and $p > 0$, $q = n$ the same arguments work, where only the rays $X_{q+1}$ and $X_{p-1}$, respectively, should be taken care of. \end{proof}

How much does the appearance of flock  depend on the choice of the anchor set $S$ of $T$?
In preparation for answering this question, we need an important general fact.

\begin{prop}\label{prop:16.12}
 Assume that $T = (X_0, \dots, X_n)$ is a direct QL-path and that $S = (Y_0, \dots, Y_m)$, $S' = (Y'_0, \dots, Y'_{m'})$ are two anchor sets of $T$. Then $m = m'$ and for any ray~ $X_p$ in $T$ the anchors of $X_p$ in $S$ correspond uniquely to the anchors of $X_p$ in $S'$. More precisely, if $Y_k$ is a legal (resp. illegal) anchor of $X_p$ in $S$,  then $Y'_k$ is a legal (resp. illegal) anchor of $X_p$ in $S'$.
\end{prop}
\begin{proof}
  Using Procedure~\ref{proc:15.5},  we obtain step by step anchor sets $(Y_0, \dots, Y_{m(i)})$ of the subpaths $(X_0, \dots, X_i)$, $0 \leq i \leq n$, with $\{ 0 \} = m(0) \leq m(1) \leq \dots \leq m(n) = m$,  where the chosen anchors of a ray $X_p$ with $p \leq i$ remain the same when proceeding from $(X_0, \dots, X_i)$ to $(X_0, \dots, X_{i+1})$. Applying this procedure a second time, we obtain anchor sets $(Y'_0, \dots, Y'_{m (i)})$, $0 \leq i \leq n$ of the same lengths $m(i)$. By induction on $i$ for $i = 0, \dots, n$ we see that the last claim of the proposition holds for all rays $X_p$ with $p \leq i$, in $(Y_0, \dots, Y_{m(i)})$ and $(Y'_0, \dots, Y'_{m (i)})$, and thus for all rays $X_p$ in $T$. \end{proof}

\textit{Comment.} This uniqueness result for anchor sets is less trivial than it may appear at first glance. Recall that a point in the down set $X^{\downarrow}$ of a ray $X$ means a QL-star contained in $\QL(X)$. There can be many such stars for fixed $X$ which are widely unrelated.

\begin{prop}\label{prop:16.13}
 Given a direct QL-path $T = (X_0, \dots, X_n)$ and an anchor $(Y_0, \dots, Y_m)$ of ~$T$, if $Y_k \in X_p^{\downarrow}$ for some $k \in [0, m]$, $p \in [0, n]$, then $Y_k$ is an anchor of $X_p$ (Definition~\ref{def:15.5}.(ii)).
\end{prop}
\begin{proof}
 We choose $q \in [0, n]$ such that $Y_k$ is the legal anchor of $X_q$. Then $X_p$ and $X_q$ are in~ $Y_k^{\uparrow}$, and thus $(X_p, X_q)$ is quasilinear by Theorem~\ref{thm:4.6}, whence $q \in \{ p-1, p+1 \}$. If $q = p$ then $X_p$ is the legal anchor of $Y_k$. If $q = p-1$ then $X_{p-1}$ and $X_p$ are in $Y_k^{\uparrow}$. Thus $(X_{p-1}, X_p)$ is a twin-pair in $T$, with anchor $Y_k$, and so $Y_k$ is an anchor of $X_p$. % (Recall again Definition~15.5.b.)
 If $q = p+1$ we conclude in the same way that $(X_p, X_{p+1})$ is a twin pair, and so $Y_k$ is an anchor of $X_p$. \end{proof}
\begin{cor}\label{cor:16.14}
 Assume that $(Y_0, \dots, Y_m)$ and $(Y'_0, \dots, Y'_m)$ are anchor sets of  a direct QL-path $(X_0, \dots, X_n)$. Let $p \in [0, n]$ and $k \in [0, m]$. Then $Y_k \in X_p^{\downarrow}$ iff $Y'_k \in X_p^{\downarrow}$.\end{cor}
\begin{proof}
 Immediate from Propositions \ref{prop:16.12} and \ref{prop:16.13}. \end{proof}

\begin{thm}\label{thm:16.15}
 Assume again that $S = (Y_0, \dots, Y_m)$ and $S'=(Y'_0, \dots, Y'_m)$ are anchor sets of a direct QL-path $T = (X_0, \dots, X_n)$. Assume further that the pair $(T, S)$ is flocky. Then $(T, S')$ is flocky.
\end{thm}
\begin{proof}
 Given $k \in [0, m-1]$ such that $Y_k^{\uparrow} \cap Y_{k+1}^{\uparrow} \ne \emptyset$ there is some $p \in [0, n]$ with $X_p \in Y_k^{\uparrow} \cap Y_{k+1}^{\uparrow}$, since $(T, S)$ is flocky. We conclude by Corollary~\ref{cor:16.14} that $X_p \in (Y'_k)^{\uparrow} \cap (Y'_{k+1})^{\uparrow}$. Thus $(T, S')$ is flocky. \end{proof}

\begin{defn}\label{def:16.16}
 In consequence of this theorem, a $\QL$-path $T$ is named a \textbf{flocky direct $\QL$-path}, if $(T, S)$ is flocky for any anchor set $S$ of $T$.
%for some anchor set $S$ of $T$, and so for every anchor set $S$ of ~$T$.
\end{defn}

Theorem \ref{thm:16.5} leads to a procedure that turns a minimal QL-path to a flocky minimal path.

\begin{defn}\label{def:16.17}
 Assume that $S = (Y_0, \dots, Y_m)$ is a subquasilinear sequence (cf. Definition~\ref{def:15.8}).
\begin{enumerate}\eroman
  \item  A \textbf{track} in $S$ is a subsequence $(Y_k, Y_{k+1}, \dots, Y_{k+t+1})$ with $t \geq 0$, $k + t < m$, such that $Y_i^{\uparrow} \cap Y_{i+1}^{\uparrow} \ne \emptyset$ for $k \leq i \leq k + t$.
\item A subsequence $A$ of $S$ is called  \textbf{trackless}, if no ray in $A$ is a member of a track in~ $S$.
\end{enumerate}
\end{defn}
\noi Note that $S$ is the disjoint union of its maximal tracks and its maximal trackless subsequences.

In this terminology Theorem \ref{thm:16.5} states that, given a maximal track $U = (Y_k, \dots, Y_{k+t+1})$ in an anchor set $S = (Y_0, \dots, Y_m)$ of a minimal QL-path $T = (X_0, \dots, X_n)$,  after choosing rays $W_i$ in $Y_i^{\uparrow} \cap Y_{i+1}^{\uparrow}$ for $k \leq i \leq t$, we obtain a new minimal QL-path
\begin{equation}\label{eq:16.5}
 T' = (X_0, \dots, X_s, W_k, \dots, W_{k+t}, X_r, \dots, X_n), \end{equation}
again of length $n$, which also has $S$ as an anchor set. If $k > 0$, then $X_s$ has the anchor $Y_k$ or $Y_{k-1}$ in $S$, but not both, since otherwise the track $U$ would not be maximal. If $k + t < m$, then for the same reason $X_r$ has the anchor $Y_{k+t}$ or $Y_{k+t+1}$, but not both. If $k = 0$, then $s = 0$, and of course, $Y_0$ is the unique anchor of $X_0$ in $S$, while if $k + t = m$ we have $r = n$, and $Y_m$ is the unique anchor of $X_n$.

It is now immediate from Theorem \ref{thm:16.11} that the subpath
\begin{equation}\label{eq:16.6}
 \Pi = (X_s, W_k, \dots, W_{k+t}, X_r) \end{equation}
of $T'$ is flocky with anchor set $U$ and has a unique maximal flock. This flock is $\Pi$ itself in the case that $X_s$ has anchor $Y_k$ and $X_r$ has anchor $Y_{k+t}$. Otherwise the maximal flock of $\Pi$ is obtained by omitting in $\Pi$ the ray $X_s$ or the ray $X_r$ or both.

\begin{defn}\label{def:16.18}
 We call $T'$ a \textbf{flock modification} of the minimal QL-path $T$ on the track~ $U$. Modifying $T$ successively on all maximal tracks in $S$ we obtain a flocky minimal QL-path $\htT$ of length $n$ with anchor set $S$, which we call a \textbf{total flock modification} of $T$.
\end{defn}

\begin{rem}\label{rem:16.19}
Given a flock modification $T' = (X_0, \dots, X_s, W_k, \dots, W_{k+t}, X_r, \dots, X_n)$ of~ $T$ on the track $U = (Y_k, \dots, Y_{k+t+1})$, we obtain all flock modifications of $T$ on $U$ by varying each ray $W_i$, $k \leq i \leq k + t$, within $Y_i^{\uparrow} \cap Y_{i+1}^{\uparrow}$. The point here is that the sets $Y_i^{\uparrow} \cap Y_{i+1}^{\uparrow}$ are convex and quasilinear in $\Ray (V)$.
\end{rem}

\section{Further observations on minimal QL-paths and their anchor sets}\label{sec:17}

Arguing similarly as often in \S\ref{sec:14}--\S\ref{sec:16}, we obtain the following facts about minimal QL-paths, mainly by exploiting Theorem~\ref{thm:4.6}.
\begin{thm}\label{thm:17.1}
 Assume that $(X_0, \dots, X_n)$ is a minimal QL-path.
\begin{itemize}
\item[a)] If $p, q \in [0, n]$ and $p+1 < q$, then the sets $(X_p^{\downarrow})^{\uparrow}$ and $(X_q^{\downarrow})^{\uparrow}$ are disjoint.
    \item[b)] If $p, q \in [0, n]$ and $p+1 = q$, the sets $(X_p^{\downarrow})^{\uparrow}$ and $X_q^{\uparrow}$ are disjoint, and the sets $X_p^{\uparrow}$ and $(X_q^{\downarrow})^{\uparrow}$ as well.
\end{itemize}
\end{thm}

\begin{proof}
 a): Let $0 \leq p < q \leq n$. Assume that there exists a ray $Z$ in $(X_p^{\downarrow})^{\uparrow} \cap (X_q^{\downarrow})^{\uparrow}$. We choose $Y_1 \in X_p^{\downarrow}$ and $Y_2 \in X_q^{\downarrow}$ with $Z \in Y_1^{\uparrow} \cap Y_2^{\uparrow}$. Then we have the diagram

\begin{equation}\label{eq:17.1}
\setlength{\arraycolsep}{0.1cm}
\begin{array}{ccccccc}
    &      &     &  \Rnode{a}{Z} &      &       &     \\[0.5cm]
\Rnode{b}{^{\dots \; \hbox to 0.6cm{\hrulefill}}} & \Rnode{c}{X_p} & \Rnode{d}{^{\hbox to 0.6cm{\hrulefill}}} &  \Rnode{e}{^{\hbox to 0.6cm{\hrulefill}}}  &  \Rnode{f}{^{\hbox to 0.6cm{\hrulefill}}}  &\Rnode{g}{X_q} & \Rnode{h}{^{\hbox to 0.6cm{\hrulefill} \; \dots}} \\[0.5cm]
    &    & \Rnode{i}{Y_1} & & \Rnode{j}{Y_2} &     &     \\
\end{array}
\psset{nodesep=5pt,arrows=->} \everypsbox{\scriptstyle}
\ncLine{a}{i} \ncLine{a}{j}
\ncline{c}{i} \ncline{g}{j}
\psset{nodesep=5pt,arrows=-} \everypsbox{\scriptstyle}
\ncLine{a}{c} \ncLine{a}{g}
\end{equation}
and obtain a QL-path $(X_0, \dots, X_p, Z, X_q, \dots, X_n)$ which will be shorter than $(X_0, \dots, X_n)$ if $q > p + 1$. Thus $Z$ does not exist in this case, which proves that $(X_p^{\downarrow})^{\uparrow}$ and $(X_q^{\downarrow})^{\uparrow}$ are disjoint for $q > p+1$.\pSkip
b): Let again $0 < p < q \leq n$. Assume that there exists some $Z \in (X_p^{\downarrow})^{\uparrow} \cap X_q^{\uparrow}$. Then we have a diagram
\begin{equation}\label{eq:17.2}
\setlength{\arraycolsep}{0.1cm}
\begin{array}{ccccccccc}
    &      &     &      \Rnode{a}{Z} &  &  &    &       &     \\[0.5cm]
\Rnode{b}{^{\dots \; \hbox to 0.6cm{\hrulefill}}} & \Rnode{c}{X_p} & \Rnode{d}{^{\hbox to 0.6cm{\hrulefill}}} & \Rnode{e}{}  & \Rnode{f}{^{\hbox to 0.6cm{\hrulefill}}}  & \Rnode{g}{X_{p+1}} &  \Rnode{h}{^{\hbox to 0.6cm{\hrulefill}}} & \Rnode{i}{X_{p+2}} & \Rnode{j}{^{\hbox to 0.6cm{\hrulefill} \; \dots}} \\[0.5cm]
    &      &   &       \Rnode{k}{Y_1} &  & &   &       &     \\
\end{array}
\psset{nodesep=5pt,arrows=->} \everypsbox{\scriptstyle}
\ncLine{a}{g}
\ncLine{a}{k}
\ncline{c}{k}
%\ncline{f}{j}
\psset{nodesep=5pt,arrows=-} \everypsbox{\scriptstyle}
\ncLine{a}{c}
\ncLine{a}{i}
\ncline{g}{k}
\end{equation}
due to the fact that $X_{p+2} \in \QL (X_{p+1}) \subset \QL (Z)$. But this is impossible since \linebreak $(X_0, \dots, X_p, Z, X_{p+2}, \dots, X_n)$ would be  a QL-path of length $n-1$. We conclude that $(X_p^{\downarrow})^{\uparrow}$ and $X_q^{\uparrow}$ are disjoint. Switching to the opposite path $(X_n, \dots, X_0)$, which again is minimal, we obtain that $X_p^{\uparrow}$ and $(X_q^{\downarrow})^{\uparrow}$ are disjoint. \end{proof}

Given a direct QL-path $T = (X_0, \dots, X_n)$ and an anchor set $S = (Y_0, \dots, Y_m)$ of $T$, we call the full subdiagram of $\hat{\Gamma}_{\QL} (V, q)$ spanned by the rays $(X_0, \dots, X_n, Y_0, \dots, Y_m)$ the \textbf{anchor diagram of the pair} $(T, S)$. Most often we display the anchor diagram as a ``bipartite'' subdiagram of $\hat{\Gamma}_{\QL} (V, q)$, with an upper horizontal row containing the rays of $T$ and a lower horizontal row containing the rays of $S$. In the lower horizontal there do not occur any arrows ($\rightarrow$) in consequence of Theorem~\ref{thm:16.8}. But there can be edges ($\raisebox{1mm}[4mm][2mm]{\hbox to 0.5cm{\hrulefill}}$)
and so QL-subpaths of the sql sequence $(Y_0, \dots, Y_n)$, at which we now take a look. First note that all QL-subpaths of $S$ are direct, as stated in Theorem~\ref{thm:15.9}.

\begin{lem}\label{lem:17.2}
 Assume that $T = (X_0, \dots, X_n)$ is a direct QL-path and $S = (Y_0, \dots, Y_m)$ is an anchor set of $T$. Assume further that $(X_p, X_{p+1})$ is a twin pair with anchor $Y_k$ in $S$. Then neither $(Y_{k-1}, Y_k)$ (if $k > 0$) nor $(Y_k, Y_{k+1})$ (if $k < m$) is ql.
 \end{lem}
 \begin{proof}
 Otherwise we have a diagram
\begin{equation}\label{eq:str}
\setlength{\arraycolsep}{0.1cm}
\begin{array}{ccccc}
\Rnode{a}{X_{p}} & \Rnode{b}{^{\hbox to 0.6cm{\hrulefill}}} & \Rnode{c}{X_{p+1}} & \Rnode{d}{^{\hbox to 0.6cm{\hrulefill}}} & \Rnode{e}{X_{p+2}}  \\[1cm]
  &  \Rnode{f}{Y_k} & & \Rnode{g}{^{\hbox to 0.6cm{\hrulefill}}} & \Rnode{h}{Y_{k+1}} \\
\end{array}
\psset{nodesep=5pt,arrows=->} \everypsbox{\scriptstyle}
\ncLine{a}{f} \ncLine{c}{f}
\ncline{e}{h} \tag{$*$}
\end{equation}
or
\begin{equation}\label{eq:sstr}
\setlength{\arraycolsep}{0.1cm}
\begin{array}{ccccc}
\Rnode{a}{X_{p-1}} & \Rnode{b}{^{\hbox to 0.6cm{\hrulefill}}} & \Rnode{c}{X_{p}} & \Rnode{d}{^{\hbox to 0.6cm{\hrulefill}}} & \Rnode{e}{X_{p+1}}  \\[1cm]
\Rnode{f}{Y_{k-1}} & \Rnode{g}{^{\hbox to 0.6cm{\hrulefill}}} & & \Rnode{h}{Y_{k}} & \\
\end{array}
\psset{nodesep=5pt,arrows=->} \everypsbox{\scriptstyle}
\ncLine{a}{f} \ncLine{c}{h}
\ncline{e}{h}\tag{$**$}
\end{equation}
But ($\ast$) would imply that $(X_p, X_{p+2})$ is ql, and ($\ast\ast$) would imply that $(X_{p-1}, X_{p+1})$ is ql, in contradiction to the assumption that $T$ is direct.
\end{proof}

The lemma has the following immediate consequence

\begin{thm}\label{thm:17.3}
 Assume again that $T = (X_0, \dots, X_n)$ is a direct QL-path and $S = (Y_0, \dots, Y_m)$ is a anchor set of $T$. Assume further that a subsequence $(Y_k, Y_{k+1}, \dots, Y_{k+t})$ of $T$ with $k \geq 0$, $k + t \leq m$, $t \geq 1$, is a QL-path, necessarily direct. Then the part of the anchor diagram of $(T, S)$ lying over this sequence is an enlargement

\begin{equation}\label{eq:17.3}
\setlength{\arraycolsep}{0.1cm}
\begin{array}{ccccccc}
\Rnode{a}{X_{p}} & \Rnode{b}{^{\hbox to 0.6cm{\hrulefill}}} & \Rnode{c}{X_{p+1}} & \Rnode{d}{^{\hbox to 0.6cm{\hrulefill}}} & \Rnode{e}{\quad} & \Rnode{f}{^{\hbox to 0.6cm{\hrulefill}}} & \Rnode{g}{X_{p+t}}  \\[1cm]
\Rnode{h}{Y_{k}} & \Rnode{i}{^{\hbox to 0.6cm{\hrulefill}}} &  \Rnode{j}{Y_{k+1}} & \Rnode{k}{^{\hbox to 0.6cm{\hrulefill}}} & \Rnode{l}{\quad} & \Rnode{m}{^{\hbox to 0.6cm{\hrulefill}}} & \Rnode{n}{Y_{k+t}} \\
\end{array}
\psset{nodesep=5pt,arrows=->} \everypsbox{\scriptstyle}
\ncLine{a}{h} \ncLine{c}{j}
\ncline{g}{n}
\end{equation}
of this subsequence, consisting of singles $X_p, X_{p+1}, \dots, X_{p+t}$ of $T$, and their (legal) anchors $Y_k, Y_{k+1}, \dots, Y_{k+t}$.
\end{thm}
\noi We call \eqref{eq:17.3} a \textbf{QL-block} (of length $t$) in the anchor diagram of $(T, S)$.

In contrast to the flocks in $T$, the subpaths of $T$ appearing as the upper horizontal of a QL-block strongly depend  on the choice of the anchor set $S$ of $T$. In particular we can choose for every single $X_p$ in $T$ as an anchor the ray $X_p$ itself, and then obtain a \textit{special anchor set} (Procedure~\ref{proc:15.5}), for which the QL-subpaths of $T$ consisting of singles give the QL-blocks of $(T, S)$; so, the maximal such subpaths of $T$ give all maximal QL-blocks of $(T, S)$.

\begin{prop}\label{prop:17.4}
 The upper horizontal of an anchor diagram $(T, S)$ of a direct QL-path $T = (X_0, \dots, X_n)$ with $n \geq 2$ does not contain arrows ($\rightarrow$) except those pointing to $X_0$ or ~$X_n$.\end{prop}

\begin{proof}
  Suppose there is an arrow, say, $X_p \to X_{p+1}$ with $p + 1 < n$. Then $X_{p+2} \in \QL (X_{p+1}) \subset \QL (X_p)$, and so the pair $(X_p, X_{p+2})$ would be quasilinear, contradicting the assumption that~ $T$ is direct. \end{proof}
\begin{defn}\label{def:17.5}
 An arrow $X_{n-1} \to X_n$ gives us a diagram
\begin{equation}\label{eq:17.4}
\setlength{\arraycolsep}{0.1cm}
\begin{array}{cccccc}
\Rnode{a}{^{\hbox to 0.6cm{\hrulefill}}} & \Rnode{b}{X_{n-2}} & \Rnode{c}{\longrightarrow} & \Rnode{d}{X_{n-1}} & \Rnode{e}{\longrightarrow} & \Rnode{f}{X_n}  \\[0.8cm]
   &    &     &    & \Rnode{g}{X_n} &
\end{array}
\psset{nodesep=5pt,arrows=->} \everypsbox{\scriptstyle}
\ncLine{d}{g} \ncLine{f}{g}
\end{equation}
and thus $(X_{n-1}, X_n)$ is a twin pair. In the same way we obtain a twin pair $(X_0, X_1)$ from an arrow $X_1 \to X_0$. We call these twin pairs $(X_{n-1}, X_n)$ and $(X_0, X_1)$ \textbf{special twin pairs}.
\end{defn}

There are other possibilities to modify a minimal QL-path $(X_0, \dots, X_n)$, $n \geq 3$, to a path $(X_0, Y, X_2, \dots, X_n)$ or $(X_0, \dots, X_{n-2}, Y, X_n)$ which starts or ends with a twin pair $(X_0, Y)$ or $(Y, X_n)$. Recall from Scholium 14.17 that, given an optimal QL-path $(X_0, \dots, X_n)$ and a ray $Y \in X_0^{\uparrow}$, the intersection $\QL (Y) \cap \{ X_0, \dots, X_n \}$ is contained in $\{ X_0, X_1, X_2 \}$. Since trivially $X_1 \in \QL (Y)$, this intersection is either $\{ X_0, X_1 \}$ or $\{ X_0, X_1, X_2 \}$.

\begin{defn}\label{def:17.6}
 We say that the optimal QL-path $(X_0, \dots, X_n)$ has \textbf{narrow entrance} if $X_2 \not\in \QL (Y)$ for every $Y \in X_0^{\uparrow}$, and that $(X_0, \dots, X_n)$ has \textbf{wide entrance} otherwise. Analogously we say that $(X_0, \dots, X_n)$ has \textbf{narrow exit}, if $\QL (Y) \cap \{ X_0, \dots, X_n \} = \{ X_{n-1}, X_n \}$ for every $Y \in X_n^{\uparrow}$, and \textbf{wide exit} otherwise. \end{defn}

If $(X_0, \dots, X_n)$ has narrow entrance, it can nevertheless happen that $X_2 \in \QL (Y)$ for a ray $Y$ in $(X_0^{\downarrow})^{\uparrow}$. In this case we have a diagram
\begin{equation}\label{eq:17.5}
\setlength{\arraycolsep}{0.1cm}
\begin{array}{ccccccccc}
    &      &   \Rnode{a}{Y} &  &  &    &   &    &     \\[0.5cm]
\Rnode{b}{X_0} & \Rnode{c}{^{\hbox to 0.6cm{\hrulefill}}} & \Rnode{d}{} & \Rnode{e}{^{\hbox to 0.6cm{\hrulefill}}} & \Rnode{f}{X_1} & \Rnode{g}{^{\hbox to 0.6cm{\hrulefill}}}  & \Rnode{h}{X_2} &  \Rnode{i}{^{\hbox to 0.6cm{\hrulefill}} \; \cdots \; ^{\hbox to 0.6cm{\hrulefill}}} & \Rnode{j}{X_n}  \\[0.5cm]
    &      &   \Rnode{k}{Z} &  &   &   &   &    &
\end{array},
\psset{nodesep=5pt,arrows=->} \everypsbox{\scriptstyle}
\ncLine{b}{k}
\ncLine{a}{k}
\psset{nodesep=6pt,arrows=-} \everypsbox{\scriptstyle}
\ncLine{a}{b}
\ncLine{a}{h}
\end{equation}
and so a new QL-path $T': = (X_0, Y, X_2, \dots, X_n)$ of length $n$. We face the problem that $T'$ perhaps is not a direct path. This problem vanishes if $T$ is minimal, since then $T'$ is again minimal and so is direct. The following is now obvious.

\begin{thm}\label{thm:17.7}
 If $T = (X_0, \dots, X_n)$ is a minimal QL-path with narrow entrance, and if $X_2 \in \QL (Y)$ for a ray $Y$ in $(X_0^{\downarrow})^{\uparrow}$, then the minimal QL-path $T' = (X_0, Y, X_2, \dots, X_n)$ arising in the diagram \eqref{eq:17.5} starts with a twin pair $(X_0, Y)$. Analogously, if $T$ has narrow exit and $W$ is a ray in $(X_n^{\downarrow})^{\uparrow}$ with $X_{n-2} \in \QL (W)$, then $T'' = (X_0, \dots, X_{n-2}, W, X_n)$ is a minimal QL-path ending with a twin pair $(W, X_n)$.
\end{thm}
\noi
These modifications $T'$ and $T''$ of $T$  have narrow entrance and narrow exit, respectively.

The following proposition indicates that minimal QL-paths with narrow entrance abound.

\begin{prop}\label{prop:17.8}
 Assume that $(X_0, \dots, X_n)$ is a minimal QL-path with $n > 3$. Then for any $r \in [1, n-1]$ the subpath $(X_r, \dots, X_n)$ is a minimal QL-path with narrow entrance.
\end{prop}
\begin{proof}
 $\{ X_r, \dots, X_n \}$ is minimal, since $\{ X_0, \dots X_n \}$ is minimal. Let $Y \in X_r^{\uparrow}$. By Scholium~\ref{schol:14.17}
\[ \QL (Y) \cap \{ X_0, \dots, X_n \} = \{ X_{r-1}, X_r, X_{r+1} \} \]
and consequently
$ \QL (Y) \cap \{ X_r, \dots, X_n \} = \{ X_r, X_{r+1} \}. $
\end{proof}

If $T = (X_0, \dots, X_n)$ is a minimal QL-path with \textit{wide entrance} and $Y$ is a ray in $Z^{\uparrow}$ for some $Z \in X_0^{\downarrow}$ with $X_2 \in \QL (Y)$, then $T' = (X_0, Y, X_2, \dots, X_n)$ is again a minimal QL-path which starts with a twin pair $(X_0, Y)$, as indicated in diagram \eqref{eq:17.5}. But now we already have a ray $W \in X_0^{\uparrow}$ at hands with $X_2 \in \QL (W)$ and so a modification $(X_0, W, X_2, \dots, X_n) = \tlT$ of ~$T$ which is a minimal QL-path starting with a twin-pair $(X_0, W)$. We meet a situation which indicates more freedom in the choice of minimal modifications of $T$ than in the case of narrow entrance. $X_0, W, Y$ are rays in the quasilinear convex set $Z^{\uparrow}$, and so every $Y' \in [W, Y]$ is a ray with $(X_0, Y') \in Z^{\uparrow}$. Further all three pairs $(W, Y)$, $(W, X_2)$, $(Y, X_2)$ are quasilinear, and so the convex hull of $\{ W, Y, X_2 \}$ is quasilinear, which implies that $(Y', X_2)$ is quasilinear for every $Y' \in [W, Y]$. Summarizing we obtain

\begin{thm}\label{thm:17.9}
 Assume that $T = (X_0, X_1, \dots, X_n)$ is a minimal QL-path which has wide entrance, and so there exists a ray $W \in X_0^{\uparrow}$ with $X_2 \in \QL (W)$. Assume further that there is given a ray $Y \in (X_0^{\downarrow})^{\uparrow}$ with $X_2 \in \QL (Y)$. Then for every $Y' \in [W, Y]$ the sequence $(X_0, Y', X_2, \dots, X_n)$ is a minimal QL-path, which starts with a twin pair $(X_0, Y')$ and has wide entrance. \end{thm}

\section{Domination of minimal QL-paths}\label{sec:18}

Recall that, given two QL-paths $T = (X_0, \dots, X_n)$, $T' = (X'_0, \dots, X'_n)$ of same length~ $n$, we say that $T'$ \textbf{dominates} $T$, and write $T \preceq_{\QL} T'$, if $X_i \preceq_{\QL} X'_i$ for every $i$ in $[0, n]$ (Definition~\ref{def:14.19}). In \S\ref{sec:14} we studied the domination relation when $T$ and $T'$ are direct or optimal. When $T'$ is minimal, we can say more.

\begin{prop}\label{prop:18.1}
 Let $T = (X_0, \dots, X_n)$ and $T'= (X'_0, \dots, X'_n)$ be QL-paths with $T \preceq_{\QL}~ T'$. If $T'$ is minimal, then $T$ is minimal.
\end{prop}
\begin{proof}
 Suppose there exists a QL-path $(X_0, Z_1, \dots, Z_r, X_n)$ with $r < n-1$. Since $\QL (X'_0) \supset \QL (X_0)$ and $\QL (X'_n) \supset \QL (X_n)$, the pairs $(X'_0, Z_1)$ and $(Z_r, X'_n)$ are quasilinear, and so $(X'_0, Z_1, \dots, Z_r, X'_n)$ is a QL-path of length $< n$, contradicting the minimality of $T'$. Thus $T$ is minimal. \end{proof}

Given a minimal QL-path $T = (X_0, \dots, X_n)$ and a QL-path $T' = (X'_0, \dots, X'_n)$ with $T \preceq_{\QL} T'$, it is trivial that $T'$ is also minimal in the case that $X'_0 = X_0$ and $X'_n = X_n$, simply since $T'$ has length $n$. More interest deserves to dominate a \textit{subpath} of $T$ by a minimal QL-path as follows.

\begin{prop}\label{prop:18.2}
 Assume that $T = (X_0, \dots, X_n)$ is a minimal QL-path and $(X'_1, \dots, X'_{n-1})$ is a QL-path (of length $n-2$) dominating the subpath $(X_1, \dots, X_{n-1})$ of $T$. Then $(X'_1, \dots, X'_{n-1})$ is minimal.\end{prop}

\begin{proof}
 The pairs $(X_0, X'_1)$ and $(X_n, X'_{n-1})$ are quasilinear, since $\QL (X'_1) \supset \QL (X_1)$ and $\QL (X'_{n-1}) \supset \QL (X_{n-1})$. Thus $T':= (X_0, X'_1, \dots, X'_{n-1}, X_n)$
is a QL-path of length $n$ from $X_0$ to $X_n$, and so is minimal. The subpath $(X'_1, \dots, X'_{n-1})$ of $T'$ is again minimal. \end{proof}

\begin{rem}\label{rem:18.3}
 Proposition \ref{prop:18.2} gives us a QL-path $T_0 = (X_1, \dots, X_{n-1})$ such that every QL-path dominating $T_0$ is minimal. Thus $T_0$ is an optimal path, all of whose enlargements are again optimal. \end{rem}

\begin{thm}\label{thm:18.4}
Assume that $T = (X_0, \dots, X_n)$ and $T'= (X'_0, \dots, X'_n)$ are minimal QL-paths with $T \preceq_{\QL} T'$ and that $S = (Y_0, \dots, Y_m)$ is an  anchor set of $T$.

\begin{itemize}
\item[a)] Then $S$ is also an anchor set of $T'$.
\item[b)] If $p \in [0, n-1]$ then $(X_p, X_{p+1})$ is a twin pair in $T$ iff $(X'_p, X'_{p+1})$ is a twin pair in ~ $T'$.
\item[c)] If $p \in [0, n]$ then $X_p$ is a single in $T$ iff $X'_p$ is a single in $T'$.
\item[d)] For each $p \in [0, n]$ have $X_p$ and $X'_p$ the same legal anchor and the same illegal anchors (if any) in $S$.
\end{itemize}
\end{thm}
\begin{proof}
  For every $i \in [0, n]$ we have the (minimal) subpaths $(X_0, \dots, X_i)$ and $(X'_0, \dots, X'_i)$ of~ $T$ and $T'$ with $(X_0, \dots, X_i) \preceq_{\QL} (X'_0, \dots, X'_i)$ and an anchor set $(Y_0, \dots, Y_{m(i)})$ of $(X_0, \dots X_i)$ with $0 = m(0) \leq m (1) \leq \dots \leq m(n) = m$. The claims a) -- d) can be verified successively for these sequences in a straightforward way by looking at Definitions~\ref{def:15.5} and \ref{def:16.6}. This proves the theorem. \end{proof}

\begin{cor}\label{cor:18.5}
 In the situation of Theorem \ref{thm:18.4} also the following holds.
\begin{itemize}
\item[a)] If $0 \leq p < p+t \leq n$, then $(X_p, \dots, X_{p+t})$ is a flock in $T$ iff $(X'_p, \dots, X'_{p+t})$ is a flock in $T'$.
\item[b)] A twin pair $(X_p, X_{p+1})$ in $T$ is isolated in $T$ iff $(X'_p, X'_{p+1})$ is isolated in $T'$.
\item[c)] The QL-path $T'$ is flocky iff $T$ is flocky.
\end{itemize}
\end{cor}
\begin{proof}
 Claims a) and b) are evident by Theorem 15.4, since flocks and isolated twin pairs can be characterized in terms of the anchors in $S$ of the involved rays, as made explicit in Scholium~\ref{schol:16.7}. The same holds for the property ``flocky'', namely $T$ is flocky iff for any $k \in [0, m-1]$ with $Y_k^{\uparrow} \cap Y_{k+1}^{\uparrow}$ not empty there is a ray $X_p$ in $T$ which has the legal anchor $Y_k$ and the illegal anchor $Y_{k+1}$, and this happens iff $X'_p$ has these legal and illegal anchors. \end{proof}

We are aware that a large part of the contents of Theorem \ref{thm:18.4} and Corollary \ref{cor:18.5} can be proved without employing our elaborate theory of anchors. In particular this holds for assertions b) and c) in Theorem~\ref{thm:18.4}.

\end{document}